\pdfoutput=1
\documentclass{amsart}
\usepackage{geometry}
\usepackage{amssymb, amsthm, amsmath, amsfonts}
\usepackage{amscd}
\usepackage{graphicx}
\usepackage{subcaption}
\usepackage{color}
\usepackage{hyperref}

\usepackage{parskip}
\newtheorem{theorem}{Theorem}[section]
\newtheorem{lemma}[theorem]{Lemma}
\newtheorem{addendum}[theorem]{Addendum}
\newtheorem{cor}[theorem]{Corollary}
\theoremstyle{definition}
\newtheorem{definition}[theorem]{Definition}
\newtheorem{remark}[theorem]{Remark}
\newtheorem{example}[theorem]{Example}

\newtheorem*{emb}{Casson's Embedding Theorem}

\title{Exotic smoothings via large $\mathbb{R}^4$'s in Stein surfaces}
\author{Julia Bennett}

\begin{document}

\maketitle

\begin{abstract}
We study the relationship between exotic $\mathbb{R}^4$'s and Stein surfaces as it applies to smoothing theory on more general open $4$-manifolds. In particular, we construct the first known examples of large exotic $\mathbb{R}^4$'s that embed in Stein surfaces. This relies on an extension of Casson's Embedding Theorem for locating Casson handles in closed $4$-manifolds. Under sufficiently nice conditions, we show that using these $\mathbb{R}^4$'s as end-summands produces uncountably many diffeomorphism types while maintaining independent control over the genus-rank function and the Taylor invariant. 
\end{abstract}

\section{Introduction}

There are relatively few techniques available for studying smoothing theory on open $4$-manifolds. It was shown by Quinn in \cite{Q} that every open $4$-manifold admits at least one smooth structure. While there are many examples that actually admit uncountably many diffeomorphism classes of smooth structures, it is still unknown if this can be expected in general. In particular, it is still conceivable that some open $4$-manifold is uniquely smoothable up to diffeomorphism. The goal of this paper is to extend existing technology for understanding exotic open $4$-manifolds by exploring the relationship between exotic $\mathbb{R}^4$'s and Stein surfaces. We construct an uncountable family of exotic $\mathbb{R}^4$'s with previously unknown properties, and we exploit these to gain a new level of control over existing invariants of exotic smoothings on open $4$-manifolds. 

Historically, methods for finding infinite families of exotic open $4$-manifolds relied heavily on information about the behavior of smoothings on $\mathbb{R}^4$. In \cite{G1}, Gompf defined the first infinite collection of exotic $\mathbb{R}^4$'s,  smooth $4$-manifolds that are homeomorphic but not diffeomorphic to $\mathbb{R}^4$. Uncountably many exotic $\mathbb{R}^4$'s were later produced by Taubes in \cite{Tau} by extending Donaldson theory and then applying the results in conjunction with Freedman's breakthroughs from \cite{F}. This was the first $4$-manifold to exhibit such uncountable behavior. Related families of smooth structures were subsequently defined on more general open $4$-manifolds by fixing some standard smooth structure and then attaching infinitely many different $\mathbb{R}^4$'s using an operation introduced by Gompf in \cite{G0} called \textit{end-summing}. Under sufficiently nice conditions, this  resulted in infinitely many diffeomorphism classes (sometimes countable and sometimes uncountable) on the same underlying $4$-manifold. For example, see the work of Bi\v{z}aca and Etnyre in \cite{BE}, Taylor's generalization in \cite{Tay}, results from Gompf in \cite{G1b}, and additional generalizations by Ding in \cite{D}, Fang in \cite{Fa}, and Gompf and Stipsicz in Section 9.3 of \cite{GS}. The hypotheses required for this procedure to succeed depended greatly on what was known about the $\mathbb{R}^4$'s being attached. In most cases, the resulting diffeomorphism types were distinguishable either by the \textit{Taylor invariant} from \cite{Tay} (taking values in $\mathbb{Z}^{\geq 0}\cup\{\pm\infty\}$) or a generalization of \textit{compact equivalence classes} from \cite{G1}, which  are two invariants that measure the complexity of embedded exotic $\mathbb{R}^4$'s in a given smooth $4$-manifold. However, little information beyond these invariants was known about the exotic open $4$-manifolds that were produced. Additionally, it seems likely that the methods for detecting exoticness were too coarse.  At this stage, however, it was not clear what additional structure could be expected from exotic $\mathbb{R}^4$'s that would be helpful for obtaining a better understanding of this construction.  

More recently, Gompf exploited the rich structure associated to \textit{Stein surfaces} to study open $4$-manifolds in \cite{G5}. A Stein surface is an open, complex $4$-manfiold that admits a proper, biholomorphic embedding into some ambient $\mathbb{C}^N$. These have associated \textit{adjunction inequalities} that place restrictions on an invariant called the \textit{genus-rank function}. This invariant was introduced in \cite{G5}, measuring the genera of smoothly embedded, homologically essential surfaces in a given smooth $4$-manifold. In \cite{G5}, exotic smooth structures on many handlebody interiors were defined by replacing standard $2$-handles with  \textit{Casson handles}. After embedding into Stein surfaces and applying the associated adjunction inequalities, the resulting diffeomorphism types were distinguished by their {genus-rank functions}. Despite the significant advancements in the study of exotic $4$-manifolds that have been facilitated by Stein surfaces, these techniques had not previously been considered in the case of open $4$-manifolds. The developments from \cite{G5} demonstrated the importance of Stein surfaces to our current setting. 

This paper constructs exotic $\mathbb{R}^4$'s while maintaining control over their relationship with Stein surfaces. Our primary interest will be in \textit{large} $\mathbb{R}^4$'s, i.e. those that contain smooth, compact,\linebreak codimension-$0$ submanifolds that do not  embed in the standard $\mathbb{R}^4$. For context, we note that both the Taylor invariant and compact equivalence class measure the size of $\mathbb{R}^4$'s. In particular, all $\mathbb{R}^4$'s with either non-zero Taylor invariant or a different compact equivalence class than the standard $\mathbb{R}^4$ are large.  There are various constructions of large $\mathbb{R}^4$'s appearing in the literature. For example, Gompf defines infinitely many large $\mathbb{R}^4$'s in \cite{G1} using topologically slice links and later constructs doubly-indexed uncountable families of large $\mathbb{R}^4$'s in both \cite{G1} and \cite{G1b} by analyzing compact equivalence classes, Bi\v{z}aca and Etnyre in \cite{BE}  and Taylor in \cite{Tay} find $\mathbb{R}^4$'s with arbitrarily large values of the Taylor invariant and sometimes uncountably many compact equivalence classes, and Freedman and Taylor produce a \textit{universal $\mathbb{R}^4$} in \cite{FT} containing all other $\mathbb{R}^4$'s as end-summands. However, the connection between large $\mathbb{R}^4$'s and Stein surface has never before been investigated. There do exist exotic $\mathbb{R}^4$'s that are contained in Stein surface, but all previously known examples are necessarily small. (Most small $\mathbb{R}^4$'s actually embed in $\mathbb{C}^2$, and Gompf produces small $\mathbb{R}^4$'s that admit Stein structures in both \cite{G2} and \cite{G4}.) 

We define the first collection of large $\mathbb{R}^4$'s that each come equipped with an embedding into a Stein surface, chosen so that they still carry the information necessary for manipulating both the Taylor invariant and compact equivalence classes.
\begin{theorem}
 There exist uncountably many large $\mathbb{R}^4$'s that each admit a smooth embedding into a Stein surface. Furthermore, these realize arbitrarily large (finite) values of the Taylor invariant by uncountably many compact equivalence classes. 
\end{theorem}
\noindent  It turns out that the diffeomorphism type of each Stein surface only depends on the Taylor invariant of the corresponding $\mathbb{R}^4$. So we've produced an uncountable family of large $\mathbb{R}^4$'s inside a countable family of Stein surfaces. It follows from  Remark 4.5 of \cite{Tay} that these $\mathbb{R}^4$'s fail to admit a handle decomposition without infinitely many $3$-handles, while Stein surfaces always admit a handle decomposition with no $3$- or $4$-handles. In particular, these $\mathbb{R}^4$'s do not admit Stein structures themselves. This theorem is an immediate consequence of Theorem \ref{inStein} and Corollary \ref{notStein}, which provide much more detail about the structure of these $\mathbb{R}^4$'s. Corollary \ref{notU} obtains a similar result in the case of infinite Taylor invariant. 

These properties should not be expected generically from large $\mathbb{R}^4$'s. To emphasize this, we produce another family that behaves similarly on the level of the Taylor invariant but has a very different relationship with Stein surfaces. 
\begin{theorem} 
There exist exotic $\mathbb{R}^4$'s that realize arbitrarily large (finite) values of the Taylor invariant and each contain a smooth, compact, codimension-$0$ submanifold that does not smoothly embed into any Stein surface.
\end{theorem} 
\noindent As we've stated it, this theorem is a slightly weakened version of Theorem \ref{notinStein}. Its proof relies on Lemma \ref{smalllemma}, which provides an explicit description of compact handlebodies that fail to embed into any Stein surface (even when we allow the embedding to be trivial on the level of second homology). 

We begin illustrating the usefulness of these new $\mathbb{R}^4$'s by analyzing the behavior of the genus-rank function under the end-sum operation. We observe that the adjunction inequality associated to Stein surfaces can be preserved while end-summing with certain large $\mathbb{R}^4$'s, even though they do not admit Stein structures themselves. 
\begin{theorem} 
If a smooth $4$-manifold is obtained by end-summing some Stein surface with an $\mathbb{R}^4$ defined by Theorem 1.1, then it still satisfies the adjunction inequality associated to this Stein surface.  
\end{theorem}
\noindent Much stronger versions of this theorem can be found in Lemma \ref{almoststein} and Corollary \ref{stein}. Conversely, end-summing with the large $\mathbb{R}^4$'s produced by Theorem 1.2 can significantly influence the genus-rank function. As seen in Example \ref{endsummingex}, this operation easily modifies Stein surfaces so that they no longer  satisfy the adjunction inequality. 

Armed with this result, we systematically study the way end-summing with large $\mathbb{R}^4$'s effects the Taylor invariant, compact equivalence classes, and the genus-rank function.  In the best case, we find that all three invariants can be controlled independently. To start, we restrict our attention to two invariants: 
\begin{theorem}
Suppose that $X$ is a open topological $4$-manifold that is the interior of an oriented, spin handlebody with all indices $\leq 2$ and $0<\beta_2(X)<\infty$. Then there are smooth structures on $X$ that realize infinitely many values of Taylor invariant but all produce the same genus-rank function, and infinitely many genus-rank functions occur in this way. Similarly, there are smooth structures on $X$ that produce infinitely many genus-rank functions but all realize the same arbitrarily large (finite) value of the Taylor invariant. 
\end{theorem}
\noindent After adding a technical requirement, we establish control over the missing invariant:
\begin{addendum}
If the standard smooth structure that $X$ inherits as a handlebody interior is compactly positive definite, then each pair of Taylor invariant and genus-rank function from this theorem is realized by smooth structures on $X$ representing uncountably many compact equivalence classes. 
\end{addendum}
\noindent We also provide similar theorems for each pair of invariants under various relaxed hypotheses. Additionally, we can sometimes apply other nice properties of these $\mathbb{R}^4$'s to produce uncountable families realizing each triple that occurs in these results. Precise statements of these different cases and many examples can be found in the second half of Section 3. 

The two contrasting families of $\mathbb{R}^4$'s produced in this paper rely on very different constructions. The definition  of the first family requires a careful study of Casson handles, which is used to sharpen a standard cut-and-paste argument for constructing large $\mathbb{R}^4$'s.  More specifically, we introduce a procedure  for locating Casson handles in closed $4$-manifolds.  While such a procedure already existed in various forms (e.g. see Casson's Embedding Theorem in \cite{C},  Quinn's Handle Straightening Theorem in \cite{Q}, and Gompf's exposition in Section 5 of \cite{G3a}), our method is the first to provide control over the topology of the Casson handles that are produced. After fixing our initial setup, we give a complete characterization of when a single Casson handle can be found without losing this extra control. We also investigate when multiple Casson handles can be located simultaneously. Our answer turns out to be a substantial extension of previous results. It is obtained by modifying Casson's original Embedding Theorem from \cite{C} using the Arf invariant associated to characteristic surfaces, which is defined by Freedman and Kirby in \cite{FK}. This can all be found in Section 2. The second family is constructed by exhausting the {universal $\mathbb{R}^4$} defined in \cite{FT} by smaller $\mathbb{R}^4$'s. After finding compact handlebodies that do not embed into any Stein surface, the existence of the necessary compact submanifolds follows from a technique in \cite{G1} for embedding compact handlebodies into $\mathbb{R}^4$'s.

\vspace{1em}
\textbf{Organization:} Section 2 provides a careful introduction to Casson handles, describes the procedure for locating Casson handles in closed $4$-manifolds, and investigates various applications of this procedure. In section 3, we define the two families of $\mathbb{R}^4$'s described above and study the implications of their different relationships to Stein surfaces. Finally, Section 4 analyzes the effect of end-summing with these two families.

\vspace{1em}
\textbf{Conventions:} All handlebodies are assumed to be \textit{self-indexing} and locally finite, which means they can be constructed in a locally finite way by simultaneously attaching all handles of a given index to the collection of handles that have strictly smaller index. Every open $4$-manifold is the interior of a handlebody meeting this description, and it follows from the appendix of \cite{G3} that these requirements do not increase the highest index handle required to describe a given open $4$-manifold as a handlebody interior. This paper will refer to $4$-dimensional Stein domains, the compact analogue to open Stein surfaces, as compact Stein surfaces. We direct the reader to Section 11 of \cite{GS} for more information about Stein surfaces, and also for the minimal amount of background pertaining to contact structures that is required for this paper.  As a convention, we will not assume that embeddings are orientation-preserving unless explicitly stated.  However, all immersion are assumed to be generic and all homotopies are assumed to be regular. The notation $X_\Sigma$ will be used to denote a smooth $4$-manifold that is obtained by equipping some topological $4$-manifold $X$ with a smooth structure $\Sigma$. We will leave out the subscript when the smooth structure is clear from context, as is the case for most of Section 2 and Section 3.

\vspace{1em}
\textbf{Acknowledgements:} The author would like to thank her advisor, Robert Gompf, for his thoughtful guidance and many insightful suggestions. Much of this work was completed while being supported by NSF Grant DMS-$1148490$.

\section{The first stage of embedded Casson handles}

The goal of this section is to introduce a procedure for constructing Casson handles in closed $4$-manifolds, designed to provide some control over the topology of the Casson handles that are produced. 

We begin with a discussion about kinky handles. Recall that a \textit{kinky handle} is obtained by performing self-plumbings on the standard $2$-handle. Suppose that $k$ is a kinky handle and $D$ is the immersed disk produced by the plumbing operation. The diffeomorphism type of $k$ is uniquely determined by the signs associated to the double points of $D$. To attach $k$ to a framed circle, we first attach the standard $2$-handle to this circle using the framing that is obtained by adding $-2\text{Self}(D)$ signed twists to the given framing. Then we perform the necessary self-plumbings, so that $k$ is now attaching to this circle in such a way that $D$ pushes off along the given framing to a disk whose algebraic intersection with $D$ vanishes. 
  
 It will be helpful to also have an alternate description of kinky handles. Observe that $B^4=D^2\times D^2$ is a regular neighborhood of the disk $D^2\times 0$ and also of its union with any collection of parallel disks $p_1 \times D^2, \ldots, p_g\times D^2$ for points $p_1, \ldots, p_g\in \text{int}(D^2)$. After equipping each disk with an orientation, this union is bounded by the oriented link shown in the first diagram in Figure 1. Attaching $1$-handles to $B^4$ as shown in the second diagram in Figure 1 produces a $4$-manifold that is a regular neighborhood of a smoothly immersed disk $D$, obtained by ambiently boundary summing $D^2\times 0$ with each of these parallel disks using the $2$-dimensional $1$-handles shown in red. The original choice of orientations defines an orientation on $D$, so we can arrange for the signed double points of $D$ to realize any desired configuration by choosing the correct number of parallel disks and orienting them appropriately. Observe that $\partial D$ is the circle in $\partial (B^4\cup (1\text{-handles}))$ shown in the final diagram in Figure 1. Thus, every kinky handle is diffeomorphic to $B^4\cup (1\text{-handles})$ and its attaching region is a regular neighborhood of $\partial D$ in $\partial (B^4\cup (1\text{-handles}))$,  provided we arrange for $D$ to have the appropriate signed double points. The $0$-framing on $\partial D$ from Figure 1 will be called the $0$-framing on this attaching region. From this perspective, attaching a kinky handle to a framed knot is equivalent to attaching $B^4\cup (1\text{-handles})$ along this attaching region by identifying its $0$-framing with the framing given on this knot. Using Casson's terminology, we will say that the \textit{frontier} of a kinky handle is the subset of its boundary obtained by removing its attaching region. 
 
This abstract description of kinky handles is useful when constructing Casson handles. In particular, we can define a set of framed circles on the frontier of any kinky handle $k$ by first fixing an identification with the final diagram in Figure 1 (preserving attaching regions) and then choosing a $0$-framed meridian of each dotted circle in this diagram. We say that any set of framed circles obtained in this way is a  \textit{collapsing set} for $k$. Attaching standard $2$-handles to $k$ along any collapsing set recovers the standard $2$-handle with its standard $0$-framing. To avoid special cases, we will treat the standard $2$-handle as a kinky handle with an empty collapsing set.   

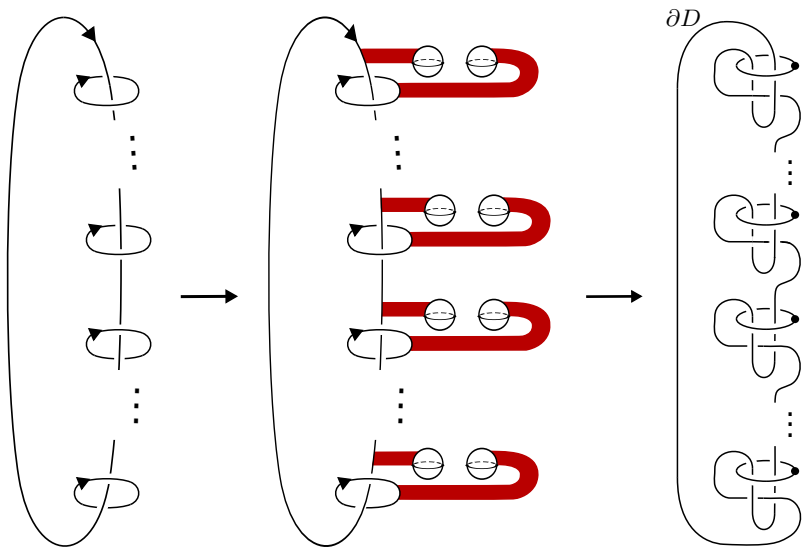
\begin{figure}
\includegraphics[scale=.7]{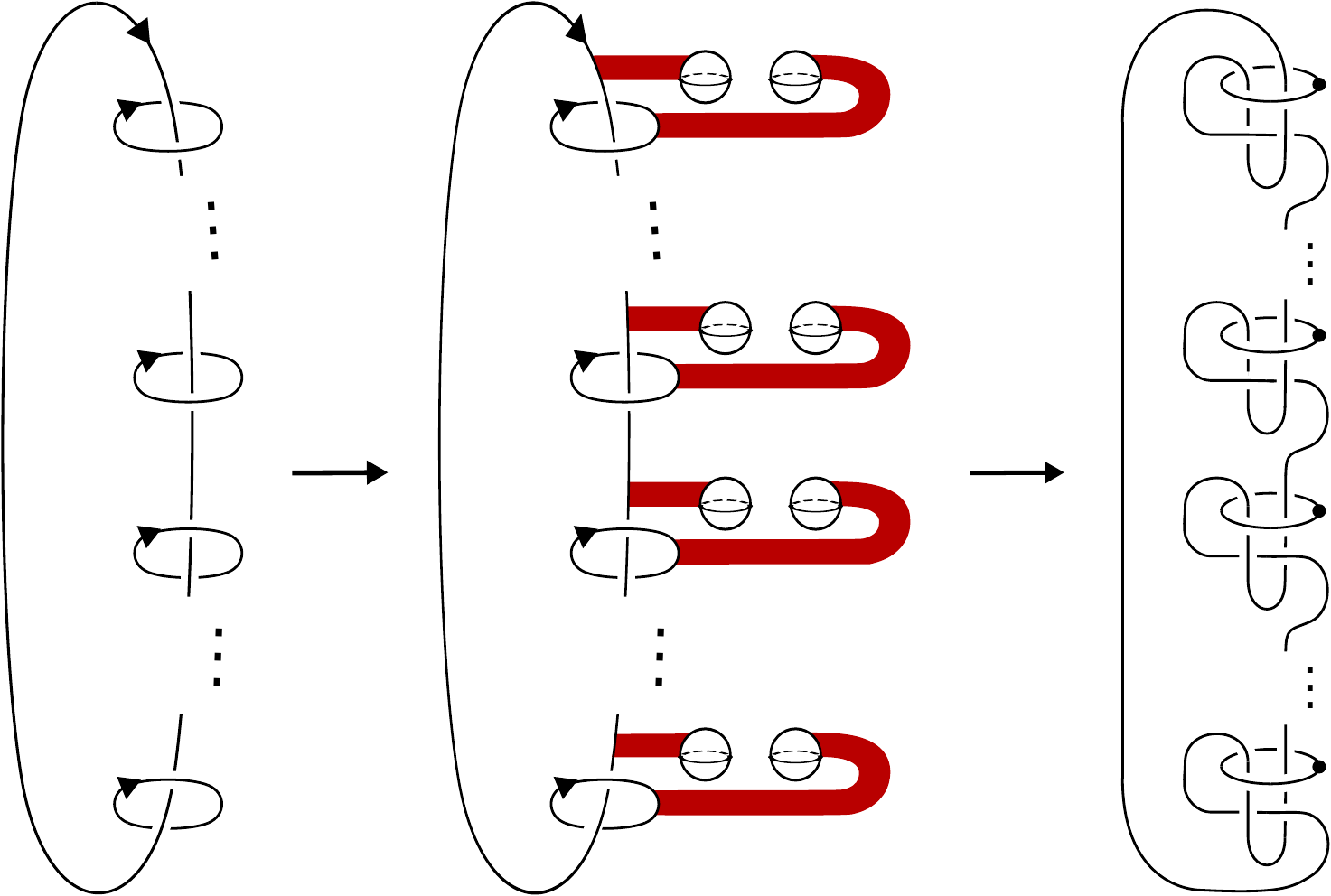}
\caption{An alternate description of kinky handles.}
\begin{picture}(0,0)
\put(98,229){\small{$\partial D$}}
\end{picture}
\end{figure}

We provide a brief description of Casson handles, but encourage the reader to see Section 2 of \cite{F}, Section 2 of \cite{G5}, or Chapter 7 of \cite{K} for alternate descriptions and more detailed exposition.  To construct a Casson handle, we start with a \textit{1-stage tower} $(T_1, \partial_-)$ consisting of a single kinky handle and its attaching region. Then we construct a \textit{2-stage tower} $T_2$ by attaching kinky handles to $T_1$ along a collapsing set for $T_1$. Next, we produce a $3$-stage tower $T_3$ by attaching kinky handles along collapsing sets for each kinky handle attached in the previous step. Iterating this procedure defines an \textit{n-stage tower} $T_n$ for every $n\in\mathbb{Z}^{>0}$. Thus, we obtain a nested sequence $T_1\subseteq T_2\subseteq T_3\subseteq \cdots$. A \textit{Casson handle} $CH$ is obtained by taking the infinite union $\bigcup_{i=1}^\infty T_i$ of any nested towers that are constructed in this manner and then removing all boundary except the interior of $\partial_-$. We refer to the immersed disk $D$ corresponding to $T_1$ as the \textit{first stage disk} of $CH$. The remaining boundary of $CH$ is the attaching region of $CH$ and it comes equipped with a $0$-framing corresponding to the $0$-framing on $\partial_-$. We can attach $CH$ to a framed circle by identifying this $0$-framing with the given framing on the circle. The fundamental result of Freedman from \cite{F} states that every Casson handle is homeomorphic to the standard open $2$-handle by a homeomorphism that preserves attaching regions. So attaching Casson handles to a collection of $0$- and $1$-handles and then removing the remaining boundary produces a smooth, open $4$-manifold that is homeomorphic (but possibly not diffeomorphic) to the result of replacing each of these Casson handles with the standard open $2$-handle. Following \cite{G5}, any smooth structure constructed in this way will be called a \textit{Casson smoothing}. If the resulting smooth $4$-manifold admits a Stein structure, then it is also referred to as a \textit{Stein-Casson smoothing}. A very useful property of Casson handles is that every Casson handle embeds into the standard $2$-handle by a smooth, orientation-preserving embedding that preserves attaching regions. 

We turn our attention to locating Casson handles in closed $4$-manifolds. We would like to realize some pre-assigned configuration of double points in the first stage disks of the Casson handles we produce. We start with a smoothly immersed sphere $S$ in a  closed $4$-manifold $X$. Observe that a regular neighborhood of $S$ in $X$ can be identified with a smoothly embedded copy of $B^4\cup T_1$ for some $1$-stage tower $T_1$ attaching to an $(S\cdot S)$-framed unknot in $\partial B^4$, with the diffeomorphism type of $T_1$ uniquely determined by the signs associated to the double points of $S$. If there are smoothly embedded, disjoint Casson handles in $X$ that ambiently attach to $B^4\cup T_1$ along a collapsing set on the frontier of $T_1$, then the interior of $B^4\cup T_1$ together with these new Casson handles is the interior of a smoothly embedded copy of $B^4\cup CH$ for some Casson handle $CH$ whose first stage is $T_1$.  We call the Casson handles attaching to $T_1$ \textit{second stage Casson handles} because they are attaching to a $1$-stage tower. So $CH$ is is obtained from $T_1$ by attaching these second stage Casson handles and then removing all remaining boundary except the interior of the attaching region of $T_1$. This construction ensures that $S$ is the union of the first stage disk of $CH$ and a smoothly embedded, unknotted disk in $B^4$. Thus, the configuration of double points in $S$ dictates the double points in the first stage disk of the resulting Casson handle $CH$. So we have reduced our task to determining conditions on smoothly immersed spheres that guarantee the existence of these second stage Casson handles. Once completed, locating a smoothly immersed sphere satisfying this criterion and realizing some desired configuration of double points is equivalent to locating a Casson handle with that configuration of double points in its first stage disk. We will later consider this setup for multiple immersed spheres in order to construct multiple Casson handles at once. 

Under sufficiently nice conditions, the proof of Casson's Embedding Theorem from \cite{C} provides a procedure for finding second stage Casson handles.  We state this theorem below using the terminology we have introduced. 
\begin{emb} Suppose that $D_1, \ldots, D_n$ are smoothly immersed disks in a smooth, simply connected $4$-manifold $Y$, with $\partial D_1, \ldots, \partial D_n$ disjointly embedded in $\partial Y$. If $D_i\cdot D_j=0$ for each $i\neq j$ and there are homology classes $\beta_1, \ldots, \beta_n\in H_2(Y)$ such that each $D_i\cdot \beta_j=\delta_{ij}$ and each $\beta_j\cdot\beta_j$ is even, then $D_1\cup \cdots \cup D_n$ can be smoothly homotoped rel $\partial$ so that $D_1, \ldots, D_n$ are the first stage disks of smoothly embedded, disjoint Casson handles $CH_1, \ldots, CH_n$ in $Y$. \end{emb}
\noindent Our primary application will only consider a single immersed disk, obtained from an immersed sphere by removing the interior of a small $B^4$ neighborhood around one of its points. However, the homotopy of $D_1\cup\cdots \cup D_n$ is not desirable because we lose control over the double points of each $D_i$. If $D_1, \ldots, D_n$ are initially disjoint and $Y-(D_1\cup\cdots \cup D_n)$ is simply connected, then it follows from the proof of this theorem that the same result holds without needing to first perform a homotopy. When $Y$ admits a spin structure, the proof also ensures that these additional hypotheses can actually replace the requirement that $\beta_1, \ldots, \beta_n$ exist and we again obtain this result without needing to perform a homotopy. In either of these two cases, the Casson handles are constructed by the method described in the previous paragraph, i.e. by identifying a regular neighborhood of $D_1\cup \cdots\cup D_n$ with disjoint $1$-stage towers and then locating disjoint second stage Casson handles in $Y$ that are ambiently attaching to this neighborhood along collapsing sets for each $1$-stage tower. A key step in the construction of these second stage Casson handles is producing immersed disks in $Y-(D_1\cup\cdots\cup D_n)$ that are bounded by these collapsing sets and each push-off by the given framing to realize an even  self-intersection number. Once these disks have been located, the existence of the second stage Casson handles only relies on the requirement that $Y-(D_1\cup \cdots\cup D_n)$ is simply connected.  Unfortunately, this proof fails to produce these disks outside of the two situations we have described. In particular, this step cannot be completed when $Y-(D_1\cup\cdots \cup D_n)$ is spin but $Y$ is not spin. This situation occurs, for example, if some $D_i$ is Poincar\`e dual to  $w_2(Y)\in H^2(Y, \mathbb{Z}_2)$. 

By using tools provided by Freedman and Kirby in \cite{FK} to choose appropriate collapsing sets, we will sometimes be able to locate these disks used in the construction of second stage Casson handles even when the conditions necessary for the proof of Casson's Embedding Theorem are not satisfied. The results from \cite{FK} associate a quadratic form $\tilde{q}:H_1(F, \mathbb{Z}_2)\to\mathbb{Z}_2$ to any smoothly embedded, orientable surface $F$ in a smooth, oriented, closed $4$-manifold $X$, provided that $F$ represents the Poincar\`e dual to $w_2(X)\in H^2(X, \mathbb{Z}_2)$. This quadratic form has the property that if $B$ is a smoothly immersed surface in $X$ bounded by a circle $a\in H_1(F, \mathbb{Z}_2)$ with the interior of $B$ disjoint from $F$, then $\tilde{q}(a)$ equals the modulo $2$ self-intersection number of $B$ obtained by pushing $B$ off along the normal framing induced by the embedding of $a$ into $F$. Theorem 1 of \cite{FK} ensures that the Arf invariant of $\tilde{q}$ equals the modulo $2$ reduction of $\frac{1}{8}(F\cdot F-\sigma(X))$.  (A nice introduction to the Arf invariant can be found in the appendix of \cite{RS}.) Unlike in \cite{FK}, we are interested in immersed spheres rather than embedded surfaces.
\begin{definition}
Let $S$ be a smoothly immersed sphere in a smooth, oriented, closed $4$-manifold $X$. We say that $(X, S)$ is a \textit{characteristic pair} if $S$ represents the Poincar\`e dual to $w_2(X)\in H^2(X, \mathbb{Z}_2)$. The \textit{Arf invariant}  of a characteristic pair $(X, S)$, denoted $\text{Arf}(X, S)$, is defined as the modulo $2$ reduction of $\frac{1}{8}(S\cdot S-\sigma(X))$. 
\end{definition}
\noindent Notice that the Arf invariant of a characteristic pair $(X, S)$ equals the Arf invariant of the bilinear form $\tilde{q}$ associated to the surface $F$ that is obtained by resolving each double point of $S$. It is precisely this relationship that will supply the disks needed in the construction of second stage Casson handles when the proof of Casson's Embedding Theorem cannot be implemented. 

Applying the methods we have introduced, we can now provide a criterion for determining if a smoothly immersed sphere in a closed $4$-manifold decomposes as the union of a smoothly embedded disk and the first stage disk of a smoothly embedded Casson handle.  
\begin{lemma} 
\label{mainlemma}
Let $S$ be a smoothly immersed sphere in a smooth, oriented, closed $4$-manifold $X$ with $X-S$ simply connected. Suppose that either $(X, S)$ is a characteristic pair with $\text{Arf}(X, S)=0$ or $(X, S)$ is not a characteristic pair. Then there is a Casson handle $CH$ such that $B^4\cup CH$ admits a smooth, orientation-preserving embedding into $X$ with $S$ equal to the union of the first stage disk of $CH$ and a smoothly embedded, unknotted disk in $B^4$, where $CH$ is attaching to $B^4$ along an $(S\cdot S)$-framed unknot in $\partial B^4$. If instead $(X, S)$ is a characteristic pair with $\text{Arf}(X, S)\neq 0$, then there is no Casson handle with such an embedding.
\end{lemma}
\noindent We note that the content of this lemma primarily comes from the characteristic cases, as the non-characteristic case is essentially a corollary to the proof of Casson's Embedding Theorem. 

\begin{proof}
We begin with the case when $(X, S)$ is not a characteristic pair, i.e. the modulo $2$ reduction of $[S]\in H_2(X)$ is not dual to $w_2(X)\in H^2(X, \mathbb{Z}_2)$. The first step is to construct an element $\beta\in H_2(X)$ with $[S]\cdot \beta=1$ and $\beta\cdot\beta$ even. Observe that $H_2(X; \mathbb{Z}_2)\cong H_2(X)\otimes \mathbb{Z}_2$ because the requirement that $X-S$ is simply connected forces $X$ to also be simply connected. Since $w_2(X)$ is uniquely characterized by the property that its pairing with any $a\in H_2(X, \mathbb{Z}_2)$ equals the square of $a$ under the standard $\mathbb{Z}_2$-pairing on $H_2(X, \mathbb{Z}_2)$, it follows that there is some $y\in H_2(X)$ with $[S]\cdot y +y\cdot y$ odd. The fact that $X-S$ is simply connected ensures the existence of a class $z\in H_2(X)$ with $[S]\cdot z=1$. If $z\cdot z$ is even, let $\beta=z$. Otherwise, let $\beta=y+(1-[S]\cdot y) z$. In either case, $[S]\cdot \beta=1$ and $\beta\cdot \beta$ is even. Now choose a small $B^4$ neighborhood around a point in $S$. Observe that $D=S-\text{int}(B^4)$ is a smoothy immersed disk in the smooth, simply connected $4$-manifold $Y=X-\text{int}(B^4)$. Since $D$ has a simply connected complement in $Y$ and $\beta$ is obviously carried in $Y$, our earlier discussion ensures that proof of Casson's Embedding Theorem now produces a smoothly embedded Casson handle $CH$ in $Y$ whose first stage disk is $D$. Hence, we have located a smoothly embedded copy of $B^4\cup CH$ in $X$ with $S$ equal to the union of the first stage disk $D$ of $CH$ and a smoothly embedded, unknotted disk in $B^4$. It is clear that $CH$ is attaching to $B^4$ along an $(S\cdot S)$-framed unknot in $\partial B^4$. 

Next, consider the case when $(X, S)$ is a characteristic pair with $\text{Arf}(X, S)\neq 0$. To obtain a contradiction, we suppose that there is a Casson handle $CH$ and smoothly embedded copy of $B^4\cup CH$ in $X$ meeting the necessary description. We can assume that $S\cdot S=1$ by blowing up an appropriate number of times and ambiently connect summing $S$ with a $\mathbb{C}P^1$ from each $\mathbb{C}P^2$ or $\overline{\mathbb{C}P^2}$ summand, ensuring that $CH$ is attaching to $B^4$ along a $1$-framed unknot in $\partial B^4$. Then the interior of $B^4\cup CH$ is homeomorphic to $\mathbb{C}P^2-\{\text{pt}\}$ and, consequently, $X$ is homeomorphic to $\mathbb{C}P^2\# Y$ for some $Y$ with $\sigma(Y)$=$\sigma(X)-1=\sigma(X)-S\cdot S$. Using the assumption that $(X, S)$ is a characteristic pair, notice that $Y$ is even because $H_2(Y)=\left<[S]\right>^\perp\subseteq H_2(X)$ and $[S]\cdot \alpha\equiv \alpha\cdot \alpha \; (\text{mod } 2)$ for every $\alpha\in H_2(X)$. Recall that the Kirby-Siebenmann invariant of a topological $4$-manifold $M$, denoted $\text{ks}(M)$, is an invariant taking values in $\mathbb{Z}_2$ that is additive under connected sum, vanishes if $M$ admits a smooth structure, and equals the modulo $2$ reduction of $\frac{1}{8}\sigma(M)$ if $M$ is even. So $\text{ks}(X)=\text{ks}(\mathbb{C}P^2)+\text{ks}(Y)=0+\text{Arf}(X, S)\neq 0$. However, this is clearly a contradiction because $X$ admits a smooth structure. 

For the remainder of the proof, we suppose that $(X, S)$ is a characteristic pair with $\text{Arf}(X, S)= 0$. The first step is to identify $\mathcal{N}(S)$ with a smoothly embedded copy of $B^4\cup T_1$ in $X$ for  some $1$-stage tower $T_1$. Let $f=S\cdot S$ and let $g$ denote the number of double points of $S$.  Consider the Kirby diagram in Figure 2, where the $1$-handle attaching spheres are identified in the usual way. Arrange for there to be exactly one $+1$-framed (resp. $-1$-framed) $2$-handle for each negative (resp. positive) double point of $S$. Instead of passing to dotted circle notation, we make sense of the framing coefficients by assuming reference arcs have been specified so that the $0$-framing on each $2$-handle is given by its blackboard framing.  Then we can fix an identification of $\mathcal{N}(S)$ with this Kirby diagram so that $S$ is obtained by pushing a spanning disk for each red curve shown on the left in Figure 3 into the $0$-handle, attaching the obvious $3g$ $2$-dimensional $1$-handles, and capping off by the core of the $f$-framed $2$-handle and two (oppositely oriented) parallel copies of the core of each $\pm1$-framed $2$-handle. Notice that removing the $f$-framed $2$-handle from $\mathcal{N}(S)$ produces a $1$-stage tower $T_1$ whose attaching region agrees with the attaching region of the missing $2$-handle. So there is a smooth, orientation-preserving embedding of $B^4\cup T_1$ into $X$ with $\mathcal{N}(S)=B^4\cup T_1$ that is defined by sending $B^4$ to this $f$-framed $2$-handle and sending $T_1$ to the union of the remaining handles. (Choosing our embedding in this way will make it easier to work with circles on the frontier of $T_1$.) Observe that $T_1$ must be attaching to $B^4$ along an $f$-framed unknot.  Also note that $S$ is the union of a smoothly embedded, unknotted disk in $B^4$ and the smoothly immersed disk associated to $T_1$. 

\begin{figure}[t]
	\includegraphics[scale=.9]{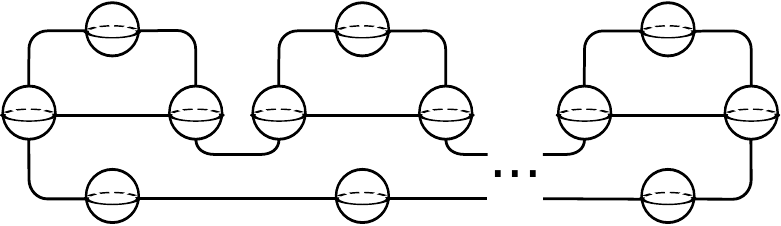}
	\begin{picture}(0,0)
    	\put(-204,48){$f$}
    	\put(-183,31){\small{$\pm1$}}
    	\put(-117,31){\small{$\pm1$}}
    	\put(-38,31){\small{$\pm1$}}
    \end{picture}  
    \caption{A Regular neighborhood of an immersed sphere.}
\end{figure}
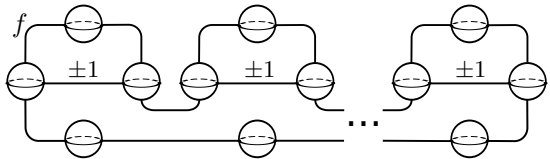

We locate the smoothly embedded, orientable surface $F$ obtained by resolving the double points of $S$ and then fix a collapsing set for $T_1$. We claim that $F$ can be seen in the Kirby diagram for $\mathcal{N}(S)$ by pushing a spanning disk for the red curve shown on the right in Figure 3 into the $0$-handle, attaching the obvious $2g$ $2$-dimensional $1$-handles, and capping off by the core of the $f$-framed $2$-handle. This follows because $S$ is obtained from the surface meeting this description by cutting out the grey annuli shown on the right in Figure 3 (pushed into the interior of the $4$-manifold to sit on the surface) and then using each $\pm 1$-framed $2$-handle to replace it with the standard model for a transverse double point of the appropriate sign.  Notice that every circle on $F$ comes with a standard framing induced by its normal framing in $F$.  It follows from our description of $F$ that there are circles $x_1, y_1, \ldots, x_g, y_g$ on $F$ that represent a symplectic basis for $H_1(F)$ and push-off to the circles $\tilde x_1, \tilde y_1, \ldots, \tilde x_g, \tilde y_g$ shown on the left in Figure 4, sending this standard framing on  $x_1, y_1, \ldots, x_g, y_g$ to  the blackboard framing on  $\tilde x_1, \tilde y_1, \ldots, \tilde x_g, \tilde y_g$. The circles $\tilde y_1, \ldots, \tilde y_g$ equipped with their blackboard framings form a collapsing set on the frontier  of $T_1$. 

Before proceeding, we study the quadratic form from \cite{FK} associated to the embedding of $F$ into $X$.  Since $F$ represents the dual to $w_2(X)\in H^2(X, \mathbb{Z}_2)$, this quadratic form $\tilde q:H_1(F, \mathbb{Z}_2)\to\mathbb{Z}_2$ is indeed well-defined. For each $i=1, \ldots, g$, let $a_i$ and $b_i$ denote the elements of $H_1(F, \mathbb{Z}_2)$ represented by $x_i$ and $y_i$ respectively. Then $a_1, b_1, \ldots, a_g, b_g$ is a basis for $H_1(F, \mathbb{Z}_2)$ with each $a_i\cdot b_j=\delta_{ij}$ under the standard pairing on $H_1(F, \mathbb{Z}_2)$.  The description in the previous paragraph ensures that $\mathcal{N}(F)$ is obtained from $\mathcal{N}(S)$ by removing the $\pm1$-framed $2$-handles from the Kirby diagram in Figure 2. So each $\tilde x_i$ lies on $\partial \mathcal{N}(F)\cap \partial \mathcal{N}(S)$ and can be isotoped on $\partial\mathcal{N}(F)$ to the attaching circle of one of these $\pm1$-framed $2$-handles, with the blackboard framing on $\tilde x_i$ mapping to the $0$-framing on this attaching circle. Thus, each $\tilde{x}_i$ bounds a smoothly embedded disk in $X-F$ that realizes an odd self-intersection number using the blackboard framing on $\tilde{x}_i$. Calling on the introduction of $\tilde{q}$ given earlier, this allows us to conclude that each $\tilde{q}(a_i)=1$ because the standard framing on $x_i$ pushes-off to the blackboard framing on $\tilde{x}_i$. (Alternatively, this can be seen using the methods of Section 8.2 in \cite{GS}.) Consequently, it follows from the appendix of \cite{RS} that $\text{Arf}(\tilde{q})$ can be computed simply by counting the number of $b_i$ with $\tilde{q}(b_i)=1$ and then reducing the result modulo $2$. On the other hand, Theorem 1 of \cite{FK} discussed earlier guarantees that $\text{Arf}(\tilde{q})=\text{Arf}(X, S)=0$. This means that the number of $b_i$ with $\tilde{q}(b_i)=1$ is even. Without loss of generality, we can therefore assume that $\tilde{q}(b_i)=1$ for $i=1, \ldots, 2k$ and $\tilde{q}(b_i)=0$ for $i=2k+1, \ldots, g$.  (Technically, this assumption changes the way that we initially identify $\mathcal{N}(S)$ with the Kirby diagram in Figure 2.) 

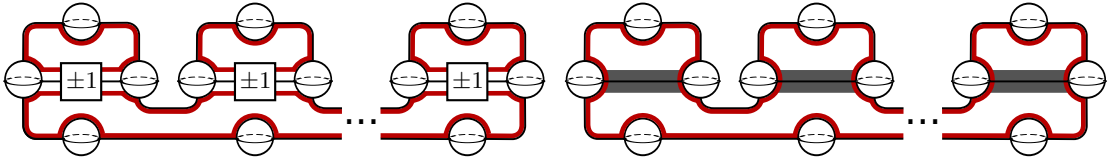
\begin{figure}[t]
  \hspace*{\fill}
  \subcaptionbox*{}{\includegraphics[scale=.9]{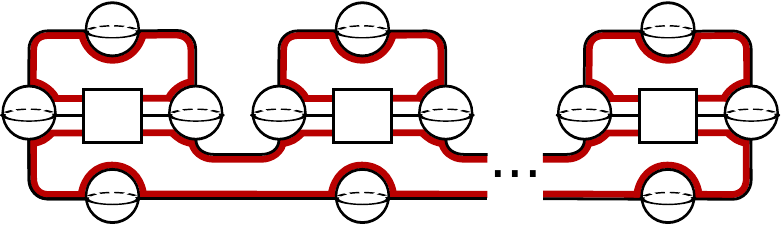}}\hfill
  \subcaptionbox*{}{\includegraphics[scale=.9]{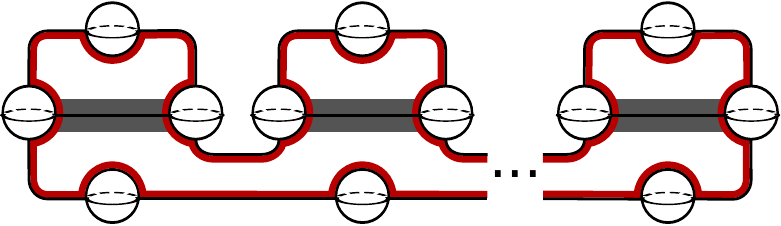}}
  \hspace*{\fill}
  \begin{picture}(0,0)
    \put(-403.5,26){\small{$\pm1$}}
    \put(-338.5,26){\small{$\pm1$}}
    \put(-259,26){\small{$\pm1$}}
  \end{picture} 
  \vspace{-1.5em}
  \caption{Resolving an immersed sphere to an embedded surface.} 
\end{figure}

Our next task is to replace the collapsing set found above with one that can be obtained by pushing-off circles on $F$ whose image under $\tilde{q}$ vanishes. Consider the $k$ disjoint circles $C_1, C_3, \ldots, C_{2k-1}$ with each $C_i$ as shown on the right in Figure 4. Each $C_i$ can be pushed into the interior of $\mathcal{N}(S)$ to lie on $S\cap F$, so that there is a self-diffeomorphism $\varphi:S\to S$ obtained by Dehn twisting on each resulting circle. This extends to a self-diffeomorphism $\varphi: \mathcal{N}(S)\to\mathcal{N}(S)$ with $\varphi(F)=F$. For each $i=1, \ldots, g$, let $y_i'=\varphi(y_i)\subseteq F$ and let $b_i'\in H_1(F, \mathbb{Z}_2)$ be the homology class it represents. The circles $\tilde y_1, \ldots, \tilde y_g$ with their blackboard framings map under $\varphi$ to  framed push-offs $\tilde y_1', \ldots, \tilde y_g'$ of the circles $y_1', \ldots, y_g'$ on $F$ with their standard framing. In particular, the framed circles $\tilde y_1', \ldots, \tilde y_g'$ also form a collapsing set on the frontier of $T_1$ because $\varphi$ maps $T_1\subseteq \mathcal{N}(S)$ to itself while keeping its attaching region fixed. Observe that each $b_i'=b_i+a_i+a_{i+1}$ for odd $i=1, \ldots, 2k-1$, each $b_{i}'=b_{i}+a_{i-1}+a_{i}$ for even $i=2, \ldots, 2k$, and each $b_i'=b_i$ for $i=2k+1, \ldots, g$. By applying the fact that $\tilde{q}$ is a quadratic form, it is now easy to verify that indeed $\tilde q(b_i')=0$ for each $i=1, \ldots, g$. 

Finally, we find the necessary Casson handle. Since $X- S$ is simply connected, there are smoothly immersed disks in $X-\text{int}(\mathcal{N}(S))$ bounded by $\tilde y_1', \ldots, \tilde y_g'$. Recall that each $\tilde{q}( b_i')=0$ and each $\tilde{y}_i'$ is a framed push-off of $y_i'$ with its standard framing. Hence, the introduction of $\tilde{q}$ given above allows us to conclude that each of these disks pushes off along the framing on $\tilde{y}_i'$ to realize an even self-intersection number.  As described earlier, the final step in proof of Casson's Embedding Theorem can now be completed to construct smoothly embedded, disjoint second stage Casson handles $CH_1, \ldots, CH_g$ in $X$ ambiently attaching to $\mathcal{N}(S)$ along $\tilde y_1', \ldots, \tilde y_g'$ using their framings. So there is a smoothly embedded Casson handle $CH$  in $X$ obtained from $T_1$ by attaching $CH_1, \ldots, CH_g$ and then removing all boundary except the interior of the attaching region of $T_1$. Therefore, we have located a smoothly embedded copy of $B^4\cup CH$ with $CH$ attaching to an $f$-framed unknot in $B^4$.  Since the first stage disk of $CH$ is the immersed disk corresponding to $T_1$, we can conclude that $S$ is equal to the union of the first stage disk of $CH$ and a smoothly embedded, unknotted disk in $B^4$.

\begin{figure}[t]
  \hspace*{\fill}
  \subcaptionbox*{}{\includegraphics[scale=.9]{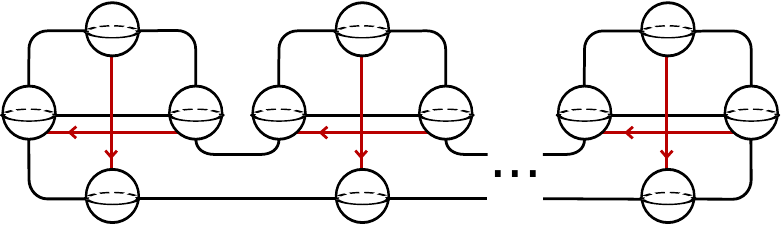}}\hfill
  \subcaptionbox*{}{\includegraphics[scale=.9]{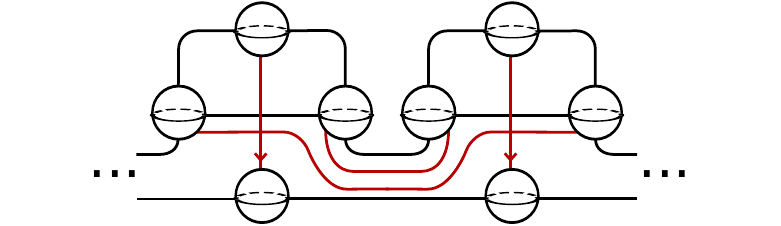}}
  \hspace*{\fill}
   \put(-391,17){\textcolor[RGB]{185,0,0}{\small{$\tilde x_1$}}}
    \put(-325,17){\textcolor[RGB]{185,0,0}{\small{$\tilde x_2$}}}
    \put(-246,17){\textcolor[RGB]{185,0,0}{\small{$\tilde x_g$}}}
    \put(-398,37){\textcolor[RGB]{185,0,0}{\small{$\tilde y_1$}}}
    \put(-333,37){\textcolor[RGB]{185,0,0}{\small{$\tilde y_2$}}}
    \put(-253.5,37){\textcolor[RGB]{185,0,0}{\small{$\tilde y_g$}}}
    \put(-73,16){\textcolor[RGB]{185,0,0}{\small{$C_i$}}}
    \put(-81,37){\textcolor[RGB]{185,0,0}{\small{$\tilde y_{i+1}$}}}
    \put(-146,37){\textcolor[RGB]{185,0,0}{\small{$\tilde y_{i}$}}}
    \vspace{-1.5em}
    \caption{Modifying the push-off of a symplectic basis for $H_1(F)$.}
\end{figure}
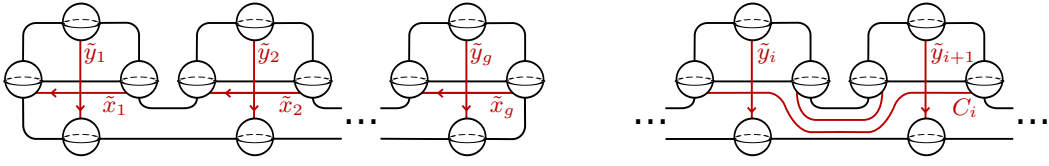
 
\end{proof}

We provide an example of the characteristic case with vanishing Arf invariant, although its usefulness will not be clear until the next section. Generalizing this example will be the first step in the proof \linebreak of Theorem \ref{inStein}. 
\begin{example}
Let $X=\overline{\mathbb{C}P^2}\#48{\mathbb{C}P^2}$, let $e_0, e_1, \ldots, e_{48}$ be the standard basis for $H_2(X)$, and let $\alpha=7e_0+e_1+\cdots + e_{48}\in H_2(X)$. It is easy to construct a smoothly immersed sphere $S$ representing $\alpha$ in $X$ that has only negative double points and has a simply connected complement. Notice that $(X, S)$ is a characteristic pair with $\text{Arf}(X, S)=0$. So Lemma \ref{mainlemma} provides a smoothly embedded copy of $B^4\cup CH$ in $X$ for some Casson handle $CH$ with only negative double points in its first stage disk, attaching to $B^4$ along a $-1$-framed unknot in $\partial B^4$. Furthermore, the topological sphere generating the second homology of $B^4\cup CH$ represents the characteristic element $\alpha\in H_2(X)$. 
\end{example}

Though we will not need more generality for our applications, we observe that this lemma can also be extended to a collection of smoothly immersed spheres to produce multiple Casson handles simultaneously. Recall that any configuration of smoothly immersed spheres in a smooth, oriented $4$-manifold has an associated graph with signed edges and weighted vertices. Conversely, every such graph corresponds to a handle decomposition for a regular neighborhood of any matching configuration of spheres. A sphere in one of these configurations has a double point if and only if there is a loop at the vertex corresponding to this sphere in the matching graph. (Further discussion about these relationships can be found in Section 6.1 of \cite{GS}.) The following corollary compiles conditions to ensure that a given configuration of smoothly immersed spheres can be ``capped off" by second stage Casson handles, as we did for a single sphere in the previous lemma.   
\begin{cor}
\label{config}
Let $S_1, \ldots, S_n$ be smoothly immersed spheres in a  smooth, oriented, closed $4$-manifold $X$ with $X-(S_1\cup\cdots\cup S_n)$ simply connected. Define $G$ to be the graph obtained by removing all loops from the graph corresponding to this configuration of spheres. Suppose that either some $(X, S_i)$ is a characteristic pair with $\text{Arf}(X, S_i)=0$ or ambiently connect summing combinations of $S_1, \ldots, S_n$ never results in a smoothly immersed sphere that forms a characteristic pair with $X$. Then there are Casson handles $CH_1, \ldots, CH_n$ such that $B^4\cup (1\text{-handles})\cup CH_1\cup\cdots \cup CH_n$ admits a smooth, orientation-preserving embedding into $X$ with each $S_i$ equal to the union of the first stage disk in $CH_i$ and a smoothly embedded disk in $B^4\cup (1\text{-handles})$, where $CH_1, \ldots, CH_n$ are attaching to $B^4\cup (1\text{-handles})$ in place of $2$-handles in the standard handle decomposition corresponding to $G$. 
\end{cor}

\begin{proof}
We start by reducing to a more familiar setting. It is clear that $\mathcal{N}(S_1\cup\cdots \cup S_n)$ can be identified with a smoothly embedded copy of $B^4\cup (1\text{-handles})\cup T_1^{(1)}\cup \cdots\cup T_1^{(n)}$, where $T_1^{(1)}, \ldots, T_1^{(n)}$ are  $1$-stage towers attaching in place of $2$-handles in the standard handle decomposition corresponding to $G$. For each $i=1, \ldots, n$, let $D_i$ be the smoothly immersed disk in $X$ associated to $T_1^{(i)}$. Also, let $N=B^4\cup (1\text{-handles})\subseteq X$.  Then each $S_i$ is the union of $D_i$ and a smoothly embedded disk in $N$. Also, each $T_1^{(i)}$ has a collapsing set on its frontier. Let   $\tilde y_1, \ldots, \tilde y_g$ be the framed circles on $\partial \mathcal {N}(S_1\cup\cdots \cup S_n)$ that form the union of these collapsing sets. So ambiently attaching smoothly embedded, disjoint second stage Casson handles in $X-\text{int}(\mathcal{N}(S_1\cup\cdots \cup S_n))$ to $\mathcal{N}(S_1\cup\cdots \cup S_n)$ along $\tilde y_1, \ldots, \tilde y_g$ produces the necessary Casson handles $CH_1, \ldots, CH_n$, where each $CH_i$ is obtained from $T_1^{(i)}$ by attaching the second stage Casson handles to its frontier and then removing all remaining boundary except the interior of the attaching region for $T_1^{(i)}$. Hence, it suffices to locate these second stage Casson handles. 

Suppose first that some $(X, S_i)$ is a characteristic pair with $\text{Arf}(X, S_i)=0$. After reordering, we assume that $(X, S_1)$ is this characteristic pair. Again after reordering, we also assume that there is a positive integer $m$ with $\tilde y_1, \ldots, \tilde y_m$ contained in $\partial \mathcal{N}(S_1)$ and $\tilde y_{m+1}, \ldots, \tilde y_{g}$ disjoint from $\partial \mathcal{N}(S_1)$. Since $X-(S_1\cup\cdots \cup S_n)$ is simply connected, there are smoothly immersed disks in $X-\text{int}(\mathcal N (S_1\cup\cdots \cup S_n))$ bounded by $\tilde y_1, \ldots, \tilde y_g$. Our discussion about Casson's Embedding Theorem ensures these second stage Casson handles exist if each disk can be chosen so that it pushes off along the given framing to realize an even self-intersection number. We apply the  proof of Theorem \ref{mainlemma} to $(X, S_1)$ to modify the circles $\tilde y_1, \ldots, \tilde y_m$ so that they still form a collapsing set for $T_1^{(1)}$ but now each disk that is bounded by one of $\tilde y_1, \ldots, \tilde y_m$ must push-off as needed. Since $S_1$ is characteristic, $X-\text{int} (\mathcal{N}(S_1))$ admits a spin structure. So $D_2, \ldots, D_n$ are smoothly immersed, disjoint disks whose union has a simply connected complement in the smooth, spin, simply connected $4$-manifold $X-\text{int}(N\cup \mathcal{N}(S_1))$. As discussed above, the proof of Casson's Embedding Theorem  allows us to assume that $\tilde y_{m+1}, \ldots, \tilde y_g$ have been chosen so that each $\tilde y_j$ in this collection bounds a disk that pushes off as needed. Thus, we can conclude that these second stage Casson handles exist. 

Next, suppose that ambiently connected summing combinations of $S_1, \ldots, S_n$ never results in a smoothly immersed sphere that forms a characteristic pair with $X$. Let $\alpha_i=[S_i]\in H_2(X)$ for each $i=1, \ldots, n$. Since $X-(S_1\cup\cdots \cup S_n)$ is simply connected, there exists $z_1, \ldots, z_n\in H_2(X)$ with each $\alpha_i\cdot z_j=\delta_{ij}$. Let $\alpha\in H_2(X)$ be obtained by summing all $\alpha_i$ for which $z_i\cdot z_i$ is odd. Notice that $\alpha$ is represented by a smoothly immersed sphere obtained by ambiently connect summing some combination of $S_1, \ldots, S_n$. As in the first paragraph of the proof of Lemma \ref{mainlemma}, our hypothesis ensures that there exists some $y\in H_2(X)$ with $\alpha\cdot y+y\cdot y$ odd. If $z_i\cdot z_i$ is even, let $\beta_i=z_i$. Otherwise, let $\beta_i=y+z_i-\sum_{j=1}^n(\alpha_j\cdot y)z_j$. It is now easy to verify that each $\alpha_i\cdot \beta_j=\delta_{ij}$ and each $\beta_j\cdot \beta_j$ is even.  Notice that every homology class in $H_2(X)$ is carried in $X-N$. In particular, each $\beta_j$ is the image of a homology class $\beta_j'\in H_2(X-N)$ under the map induced by inclusion $X-N\hookrightarrow X$. So the smoothly immersed, disjoint disks $D_1, \ldots, D_n$ have a simply connected complement in the smooth, simply connected $4$-manifold $Y=X-\text{int}(N)$ and there are homology classes $\beta_1', \ldots, \beta_n'\in H_2(Y)$ with each $D_i\cdot \beta_j'=\delta_{ij}$ and each $\beta_j'\cdot \beta_j'$ even. As discussed above, the proof of Casson's Embedding Theorem now ensures that the necessary second stage Casson handles exist after possibly modifying the collection $\tilde y_1, \ldots, \tilde y_g$.  
\end{proof}

As an example of how this corollary can be applied, we look at spheres generating the homology of nuclei in elliptic surfaces. 
\begin{example}
Fix an even, positive integer $n$ and let $N(n)$ denote the nucleus in the elliptic surface $E(n)$.  Recall that $H_2(N(n))=\mathbb{Z}\oplus \mathbb{Z}$, generated by $\alpha_1, \alpha_2\in H_2(N(n))$ satisfying $\alpha_1^2=0, \alpha_2^2=-n$, and $\alpha_1\cdot\alpha_2=1$. There is a fishtail fiber $S_1$ in $N(n)$ and a section $S_2$ in $N(n)$ representing $\alpha_1$ and $\alpha_2$ respectively. In particular, $S_1$ is a smoothly immersed sphere with a unique double point, $S_2$ is a smoothly embedded sphere, $S_1$ and $S_2$ intersect exactly once, and $E(n)-S_1\cup S_2$ is simply connected. None of $S_1, S_2,$ or $S_1\#S_2$ form a characteristic pair with $X$ because $E(n)$ is spin and inclusion sends $H_2(N(n))$ to a direct summand of $H_2(E(n))$. For any Casson handle $CH$, define $Y_{CH}=B^4\cup CH\cup h$ with $CH$  attaching to a $0$-framed unknot in $\partial B^4$ and a standard $2$-handle $h$ attaching to a $-n$-framed meridian of this unknot. Since we've verified the necessary hypotheses, Corollary \ref{config} ensures that some $Y_{CH}$ smoothly embeds into $E(n)$ in such a way that $S_1$ and $S_2$ are obtained by capping off smoothly embedded disks in $B^4$ by the first stage disk of $CH$ and the core of $h$. In particular, $Y_{CH}$ generates the homology of $N(n)$. Additionally, there is a topologically embedded sphere in $E(n)$ representing $\alpha_1$ and intersecting the section in $N(n)$ exactly once. In the $n=2$ case, we can use the discussion from Gompf and Mrowka in \cite{GM} to locate three disjoint copies of $N(2)$ in the $K3$ surface. So we obtain a pair consisting of a fishtail fiber and a section in each of these three copies of $N(2)$. Gompf (unpublished) has also observed that the union of these three pairs has a simply connected complement in the $K3$ surface. (Each nucleus contains a second fishtail fiber, and it is routine to use this to produce the nullhomotopies of the necessary meridians.) Then Corollary \ref{config} allows us to choose Casson handles $CH_1, CH_2,$ and $CH_3$ so that each has exactly one double point in its first stage disk and the disjoint union of $Y_{CH_1}, Y_{CH_2}$, and $Y_{CH_3}$ smoothly embeds into the $K3$ surface to generate its hyperbolic summand. It follows from Furuta's 10/8 Inequality in \cite{F} that it is impossible to replace any one of $CH_1, CH_2$, or $CH_3$ with a standard $2$-handle and still find such an embedding.  So we have realized the smallest number of kinks possible in the first stage disk of each of these Casson handles. This discussion can be applied to sharpen Freedman's construction of the first large $\mathbb{R}^4$, which is described in \cite{G0}. Although it not clear at this stage what properties might be useful, it is likely that this extra control could provide $\mathbb{R}^4$'s contained in connected sums of $S^2\times S^2$'s with some additional structure. On the other hand, this corollary certainly does not apply if we instead choose $n$ to be odd. In this case, the second paragraph in the proof of Lemma \ref{mainlemma} ensures that $\alpha_1$ cannot be represented by any topologically embedded sphere in $E(n)$ because $\alpha_1$ is a characteristic element in $H_2(E(n))$ and $\sigma(E(n))-\alpha_1^2=-8n$ is not divisible by $16$.
\end{example}

Given a configuration of smoothly immersed spheres, it is generally difficult to determine when it is not possible to find Casson handles that admit an embedding as in the previous corollary. There are some obvious examples where these Casson handles cannot exist because taking an ambient connected sum of spheres in the configuration would lead to a contradiction of Lemma \ref{mainlemma}. However, there are also examples where such a contradiction does not arise but the necessary Casson handles still fail to exist. This is illustrated in our final example of this section.  
\begin{example}
Let $Y_1$ and $Y_2$ each be a copy of $\mathbb{C}P^2\#8\overline{\mathbb{C}P^2}$ and let $X=Y_1\#Y_2$. Also let \linebreak $f_1, e_1, \ldots, e_8, f_2, e_9, \ldots, e_{16}$ be the standard basis for $H_2(X)=H_2(Y_1)\oplus H_2(Y_2)$. Then there are smoothly immersed spheres $S_1$ and $S_2$ in $X$ with $S_1$ representing $3f_1+e_1+\cdots +e_8$ and $S_2$ representing\linebreak $3f_2+e_9+\cdots +e_{16}$. These can be chosen to each have exactly one double point and so that each $S_i$ is contained in the $Y_i$ summand with $Y_i-S_i$ simply connected. Notice that $X-(S_1\cup S_2)$ is also simply connected. Let $F_1$, $F_2$, and $F$ be the surfaces  obtained by resolving the double points of $S_1, S_2$, and $S_1\# S_2$, respectively. Since $(Y_1, F_1)$, $(Y_2, F_2)$, and $(Y, F)$ are all characteristic pairs, the corresponding quadratic forms $\tilde{q}_1:H_1(F_1, \mathbb{Z}_2)\to \mathbb{Z}_2$, $\tilde{q}_2:H_1(F_2, \mathbb{Z}_2)\to \mathbb{Z}_2$, and $\tilde{q}:H_1(F, \mathbb{Z}_2)\to \mathbb{Z}_2$ from Corollary 1\linebreak of \cite{FK} are each well-defined. Theorem 1 of \cite{FK} guarantees that $\text{Arf}(\tilde{q}_1)=\text{Arf}(\tilde{q}_2)=\frac{1}{8}(1-(-7))=1$. After observing that $\tilde{q}=\tilde{q}_1\oplus\tilde{q}_2$,  reversing the argument used in the proof of Lemma \ref{mainlemma} ensures that there are not Casson handles $CH_1$ and $CH_2$ corresponding to this configuration as in Corollary \ref{config}.\linebreak However, $S_1$ and $S_2$ are smoothly immersed spheres in $X$ whose union has a simply connected\linebreak complement, neither $S_i$ is characteristic, and $(X, S_1\#S_2)$ is a characteristic pair with vanishing Arf invariant. 
\end{example}

\section{Large $\mathbb{R}^4$'s in Stein surfaces}

We now define a family of exotic $\mathbb{R}^4$'s that realize arbitrarily large (finite) values of the Taylor invariant and each admit a smooth, orientation-preserving embedding into a compact Stein surface. To provide contrast, we also present a family of exotic $\mathbb{R}^4$'s that fail to embed into any Stein surface and investigate the implications of the existence of these two families. 

We begin with an introduction to the Taylor invariant, first defined by Taylor in \cite{Tay}. Suppose that $X$ is a smooth, orientable, open $4$-manifold. Let $\mathcal{E}(X)$ denote the set of flat, topological embeddings $e:B^4\hookrightarrow X$ that restrict to a smooth embedding around some point $p\in \partial B^4$. For each $e\in\mathcal{E}(X)$, fix $\mathcal{S}{p}(e)$ to be the collection of smooth, closed, spin $4$-manifolds with hyperbolic intersection form that contain a smoothly embedded copy of the interior of $e(B^4)$. Roughly speaking, the Taylor invariant is computed by choosing the smallest second betti number appearing in $\mathcal{S}p(e)$ for each $e\in\mathcal{E}(X)$, diving each in half, and then taking the supremum of the resulting integers. However, it turns out that this definition fails to provide a meaningful invariant when $X$ is not spin. In order to address this issue, recall (e.g. see \cite{L}) that $X$ has associated homology groups $H_n^{lf}(X, \mathbb{Z}_2)$ and cohomology groups $H_c^n(X, \mathbb{Z}_2)$ .  These groups satisfy the duality relations $H_n^{lf}(X, \mathbb{Z}_2)\cong H^{4-n}(X, \mathbb{Z}_2)$ and $H^n_c(X, \mathbb{Z}_2)\cong H_{4-n}(X, \mathbb{Z}_2)$. Given an element in $H^2(X, \mathbb{Z}_2)$, we will say that its dual is the corresponding element in $H_2^{lf}(X, \mathbb{Z}_2)$ from the first duality isomorphism. An element in $H^2(X, \mathbb{Z}_2)$ is said to be \textit{compactly supported} if its dual is represented by a smoothly embedded, closed surface in $X$.  Since the standard maps $H_n(X, \mathbb{Z}_2)\to H_n^{lf}(X, \mathbb{Z}_2)$ and $H_c^n(X, \mathbb{Z}_2)\to H^n(X, \mathbb{Z}_2)$ commute with these duality isomorphisms in the obvious way, an element of $H^2(X, \mathbb{Z}_2)$ is compactly supported precisely when the original cohomology class pulls back to an element in $H_c^2(X, \mathbb{Z}_2)$. (This whole discussion also holds with coefficients in $\mathbb{Z}$ or any field.) If $F$ is a smoothly embedded, closed surface in $X$ that represents the dual to $w_2(X)\in H^2(X, \mathbb{Z}_2)$, then $X-F$ is spin. We can now give a complete, precise definition of the Taylor invariant. 
 \begin{definition}Suppose that $X$ is a smooth, orientable $4$-manifold. If $X$ is spin, then the Taylor invariant of $X$ is defined as 
\[\gamma(X)=\sup_{e\in\mathcal{E}(X)}\left\{\min_{N\in\mathcal{S}p(e)}\left\{\frac{1}{2}\beta_2(N)\right\}\right\}.\]
If $X$ is not spin but $w_2(X)$ is compactly supported, then the Taylor invariant of $X$, still denoted $\gamma(X)$, is the supremum of $\gamma(X-F)-\dim_{\mathbb{Z}_2}H_1(F, \mathbb{Z}_2)$ with $F$ ranging over all smoothly embedded, closed surfaces in $X$ that represent the dual to $w_2(X)$. If $w_2(X)$ is not compactly supported, then the Taylor invariant of $X$ is defined to be $\gamma(X)=-\infty$.  
\end{definition}
\noindent It follows from this definition that the Taylor invariant takes values in $\mathbb{Z}\cup\{\pm\infty\}$. It is easy to verify that $\mathcal{S}p(e)$ is non-empty for any $e\in \mathcal{E}(X)$ when $X$ is spin, ensuring that this invariant is indeed well-defined. Notice that it does not depend on a choice of orientation. 

All of the $\mathbb{R}^4$'s presented in this section are large because they have non-zero Taylor invariant. We remark that it is still unknown if there exists a large $\mathbb{R}^4$ admitting a Stein structure or a small $\mathbb{R}^4$ not admitting a Stein structure. Remark 4.5 of \cite{Tay} ensures that all Stein $\mathbb{R}^4$'s have vanishing Taylor invariant, as any $\mathbb{R}^4$ with non-vanishing Taylor invariant has infinitely many $3$-handles in any handle decomposition. So the former would provide the first example of a large $\mathbb{R}^4$ with vanishing Taylor invariant. In particular, our examples do not admit Stein structures. As mentioned above, small exotic $\mathbb{R}^4$'s admitting Stein structures are produced by Gompf in \cite{G2} and \cite{G4}.

An important tool for building open $4$-manifolds is an operation called end-summing. This operation is essentially the non-compact analogue to boundary summing. We present the definition of end-summing that is given by Gompf in  \cite{G1}, but note that there are multiple equivalent definitions appearing in the literature. 
\begin{definition}
Suppose that $X_1$ and $X_2$ are smooth, oriented, non-compact $4$-manifolds. Let \linebreak $\gamma_1:[0, \infty)\to X_1$ and $\gamma_2:[0, \infty)\to X_2$ be smooth, properly embedded rays with tubular neighborhoods $\nu_1$ and $\nu_2$, respectively. The $\textit{end-sum}$ of $X_1$ and $X_2$ along $\gamma_1$ and $\gamma_2$ is defined as 
\[X_1\cup_{\varphi_1} I\times \mathbb{R}^3\cup_{\varphi_2} X_2, \]
where $\varphi_1:[0, \frac{1}{2})\times \mathbb{R}^3\to \nu_1$ and $\varphi_2:(\frac{1}{2}, 1]\times \mathbb{R}^3\to \nu_2$ are orientation-preserving diffeomorphisms that respect the $\mathbb{R}^3$-bundle structures. 
\end{definition}
\noindent The resulting diffeomorphism type generally depends on the rays $\gamma_1$ and $\gamma_2$, but does not depend on any other choices made in the definition. Notice that it is important to fix an orientation on both $4$-manifolds before end-summing, or else the operation is not well-defined. When end-summing with open Stein surfaces, we will always assume that their complex orientations have been fixed. It follows from \cite{CG} that end-summing Stein surfaces together produces another Stein surface.   

End-summing was originally defined only for $\mathbb{R}^4$'s, and its behavior is simpler in this case. End-summing a smooth, oriented, open $4$-manifold $X$ with some (possibly exotic) oriented $\mathbb{R}^4$ produces a smooth, oriented $4$-manifold $X'$ that is homeomorphic (but possibly not diffeomorphic) to $X$. Furthermore, this homeomorphism $X'\to X$ can be chosen so that it restricts to an embedding on $X$ that is isotopic to the identity. Using any homeomorphism meeting this description, we can uniquely (up to isotopy) transport the smooth structure of $X'$ to a smooth structure on the underlying topological $4$-manifold of $X$. In summary, end-summing a smooth, oriented, open $4$-manifold with an oriented $\mathbb{R}^4$ defines a new smooth structure on the same underlying topological $4$-manifold. The resulting smooth $4$-manifold is oriented by the orientation on this underlying topological $4$-manifold, and this agrees with the orientation induced by $X'$. It follows from the appendix of \cite{G1} that the isotopy class of the smooth structure produced by this operation does not depend on the choice of ray used in the $\mathbb{R}^4$.   Similarly, simultaneously end-summing $X$ with multiple $\mathbb{R}^4$'s uniquely defines an isotopy class that only depends on the choice of rays in $X$. A slightly stronger statement from that appendix  ensures the diffeomorphism type of an oriented $\mathbb{R}^4$, which we will denote by either  $\natural_{i=1}^n R_i$ or $R_1\natural\cdots\natural R_n$, that results from simultaneously end-summing the standard $\mathbb{R}^4$ with each oriented (possibly exotic) $\mathbb{R}^4$ from some (possibly infinite) collection $\{R_i\}_{i=1}^n$ is independent of any choice of rays. When $n$ is finite, $\natural_{i=1}^n R_i$ is diffeomorphic to the result of inductively end-summing $R_i$  with $R_{i-1}$ for each $1<i\leq n$. We will often use $\natural_n R$ to denote $\natural_{i=1}^n R_i$ when each $R_i$ is a copy of $R$.

In order to end-sum with exotic $\mathbb{R}^4$'s in ambient $4$-manifolds, we will require the definition of shaved embeddings from \cite{G1}:
\begin{definition}
Let $R$ be an $\mathbb{R}^4$ and let $X$ be any smooth $4$-manifold. We say that a smooth embedding $i:R\hookrightarrow X$ is shaved if $i(R)=\text{int}(B)$ for some flat, topological $4$-ball $B$ in $X$ with a (non-empty) subset $U\subseteq \partial B$ that is a smooth, codimension-$1$ submanifold of $X$. Given orientations on $R$ and $X$, we say that $R$ is a \textit{shaved $\mathbb{R}^4$ in $X$} if $R\subseteq X$ and inclusion defines a smooth, shaved, orientation-preserving embedding of $R$ into $X$.
\end{definition}
\noindent If $R_1$ and $R_2$ are shaved $\mathbb{R}^4$'s in smooth, connected, oriented, open  $4$-manifolds $X_1$ and $X_2$ respectively, then it easy to verify that $R_1\natural R_2$ admits a smooth, shaved, orientation-preserving embedding into any end-sum of $X_1$ and $X_2$. Similarly, $R_1\natural R_2$ admits a smooth, shaved, orientation-preserving embedding into $X_1\#X_2$. (For example, these claims follow from the solution to Exercise 9.4.8(a) in \cite{GS}). The same discussion holds for finite iterated end-sums. These observations also extend in the obvious way to the case of infinite end-sums, provided they are all performed simultaneously as described above.  
 
We are prepared to define the family of exotic $\mathbb{R}^4$'s described earlier, providing some additional structure as well. The usefulness of this extra structure will be illustrated in the next section, but it is worth noting that either of the final two properties below ensures each of these $\mathbb{R}^4$'s contains an uncountable family of distinct $\mathbb{R}^4$'s sharing its Taylor invariant. Our construction is closely related to Example 5.10 of \cite{Tay}, which finds $\mathbb{R}^4$'s that satisfy properties (\ref{inv}), (\ref{end}), and (\ref{neg}) below but may not satisfy property (\ref{embeddings}). In particular, that example provides no information about the relationship between the resulting $\mathbb{R}^4$'s and Stein surfaces. 
\begin{theorem}
\label{inStein} 
There are exotic $\mathbb{R}^4$'s realizing arbitrarily large (finite) values of the Taylor invariant that each admit a smooth embedding into a compact Stein surface. More specifically, there is a smooth, oriented $4$-manifold $L$ homeomorphic to $\mathbb{R}^4$ that satisfies the following properties: \begin{enumerate}
\item  \label{inv} $\gamma(L)$ is an arbitrarily large integer, each $\gamma(\natural_nL)<\infty$ for $n\in\mathbb{Z}^{> 0}$, and $\lim_{n\to\infty}\gamma(\natural_n L)=\infty$, 
\item \label{embeddings} there is a compact, connected Stein surface $S$ such that each $\natural_n{L}$ for $n\in\mathbb{Z}^{> 0}\cup\{\infty\}$ admits a smooth, shaved, orientation-preserving embedding into an open Stein surface $S_n$ obtained simultaneous end-summing the standard $\mathbb{R}^4$ with $n$ copies of the interior of $S$, 
\item \label{end}there is a smooth, oriented, open, simply connected $4$-manifold $Z_n$  for each $n\in\mathbb{Z}^{>0}$ with an even, positive definite, unimodular intersection form such that $Z_n -\text{int} (Y_n) $ is orientation-preserving diffeomorphic to $\natural_n L - \text{int} (K_n) $ for some smooth, compact, codimension-$0$ submanifolds $Y_n \subseteq Z_n $ and $K_n \subseteq \natural_n L $, and
\item \label{neg} each $\natural_n{L}$ for $n\in\mathbb{Z}^{> 0}$ admits a smooth, shaved, orientation-preserving embedding into $\#n\overline{\mathbb{C}P^2}$.
\end{enumerate}
\end{theorem}

\begin{proof}
We construct $L$ using a standard cut-and-paste argument, but apply Lemma \ref{mainlemma} to maintain control over the first stage of the Casson handle located during this construction. If $d$ is any integer with $d>1$ and $d\equiv\pm 1 (\text{mod } 8)$, then there exists a positive integer $s$ with $d^2-1=16s$. We start by fixing any pair $(d, s)$ of positive integers satisfying this equality. Let $X=\overline{\mathbb{C}P^2}\#16s{\mathbb{C}P^2}$ and let $e_0, e_1, \ldots, e_{16s}$ be the standard basis for $H_2(X)$. It is easy to construct a smoothly immersed sphere $S$ in $X$ representing the element $\alpha=de_0+e_1+\cdots +e_{16s}\in H_2(X)$ that has $g=\frac{1}{2}(d-1)(d-2)$ negative double points and no positive double points. This can be chosen so that there is  a smoothly embedded sphere in $X$ intersecting $S$ exactly once, ensuring that $X-S$ is simply connected. Observe that $(X, S)$ is a characteristic pair with $\text{Arf}(X, S)=0$ because 
\[\frac{1}{8}(S\cdot S-\sigma(X))=\frac{1}{8}((-d^2+16s)-(-1+16s))=-\frac{1}{8}(d^2-1)=-2s.\]
So $S$ satisfies the hypotheses of Lemma \ref{mainlemma} with $S\cdot S=-d^2+16s=-1$. Thus, there is a neighborhood $W\subseteq X$ of $S$ that is equal to the the interior of the smoothly embedded copy of $B^4\cup CH$, where $CH$ is a Casson handle attaching to a $-1$-framed unknot in $\partial B^4$ whose first stage disk has has $g$ negative double points and no positive double points. Let $T_1$ denote the $1$-stage tower of $CH$. Notice that $W$ is homeomorphic to $\overline{\mathbb{C}P^2}-\{\text{pt}\}$. In particular, there is a flat, proper, topological embedding $j: S^3\times[0,\infty)\hookrightarrow W$. We call on Quinn's Stable Homeomorphism Theorem (Theorem 8.1A of \cite{FQ}) to arrange for $j$ to be smooth on a neighborhood of $p\times (0,\infty)\subseteq S^3\times[0,\infty)$ for some point $p\in S^3$. Define $C$ to be the subset of $W$ obtained by removing $j(S^3\times(a, \infty))$ for some $a\in\mathbb{R}^{>0}$, and then choose $C'$ to be any smooth, compact, codimension-$0$ submanifold of $W$ containing $C$. Next, recall that the standard $2$-handle can be written as $T_1\cup (2\text{-handles})$ with these standard $2$-handles attaching to any collapsing set on the frontier of $T_1$. So there is a smooth, orientation-preserving embedding of $CH$ into the standard $2$-handle that sends the interior of $T_1$ in $CH$ to the interior of $T_1$ in this description. Using the standard handle decomposition for $\overline{\mathbb{C}P^2}$, this extends to a smooth, orientation-preserving embedding $i: W\hookrightarrow \overline{\mathbb{C}P^2}$ that induces an isomorphism on the level of second homology.  As illustrated in Figure 5, we can now define $(L, K)= (\overline{\mathbb{C}P^2}-i(C),  \overline{\mathbb{C}P^2}-\text{int}(i(C')))$. Notice that $L$ and $K$ each inherit both a smooth structure and an orientation from $\overline{\mathbb{C}P^2}$. It easily follows from the Mayer-Vietoris Sequence and Seifert-van Kampen Theorem that $H_2(L)=0$  and $\pi_1(L)=1$. By construction, $L$ is simply connected at infinity because its unique end is homeomorphic to $S^3\times \mathbb{R}$. Hence, Freedman's criterion from Corollary 1.2 in \cite{F} ensures that $L$ is indeed homeomorphic to $\mathbb{R}^4$. It is clear from our construction that $L$ is a shaved $\mathbb{R}^4$ in $\overline{\mathbb{C}P^2}$.

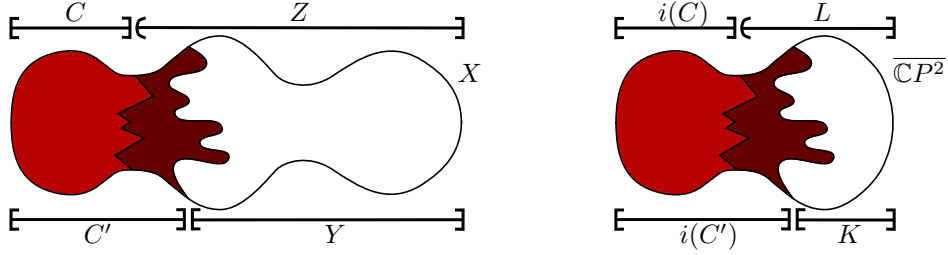
\begin{figure}[t]
  \hspace*{\fill}
  \subcaptionbox*{}{\includegraphics[scale=.75]{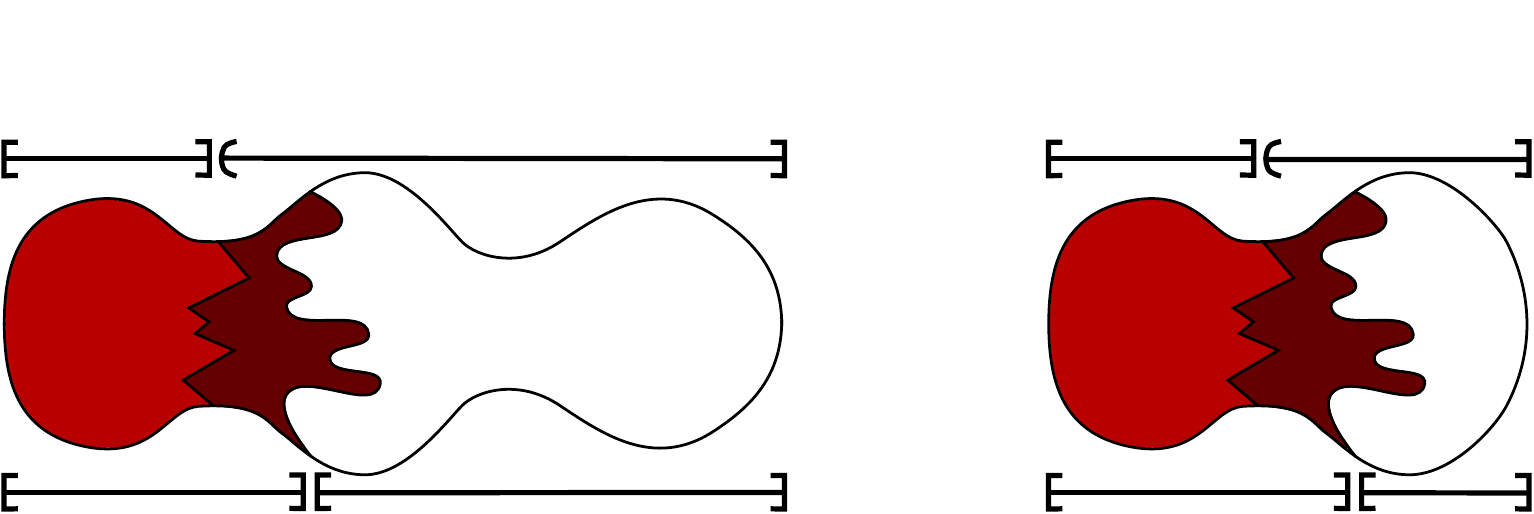}}\hfill
  \hspace*{\fill}
  \vspace{-1em}
  \caption{Constructing a large $\mathbb{R}^4$ in $\overline{\mathbb{C}P^2}$.}
  \begin{picture}(0,0)
    \put(-15,97){$X$}
    \put(148,97){$\overline{\mathbb{C}P^2}$}
    \put(-162,120){$C$}
    \put(-78,120){$Z$}
    \put(60,120){$i(C)$}
    \put(118,120){$L$}
    \put(-155,36){$C'$}
    \put(-65,36){$Y$}
    \put(68,36){$i(C')$}
    \put(127,36){$K$}
    \end{picture}
\end{figure}
 
We will work backwards to verify that properties (1)-(4) are satisfied. Since $L$ is a shaved $\mathbb{R}^4$ in $\overline{\mathbb{C}P^2}$, it is immediate from the discussion after our definition of end-summing that property (\ref{neg}) holds.  Next, we consider property (3). Let $(Z, Y) = (X-C, X-\text{int}(C'))$, so that $Z$ and $Y$ each inherit both a smooth structure and orientation from $X$. Applying the Mayer-Vietoris Sequence and Seifert-van Kampen Theorem again, we find that $H_2(Z)=\left<\alpha\right>^\perp$ and $\pi_1(Z)=1$. Since $\alpha$ is characteristic with $\alpha\cdot\alpha=-1$,  we can conclude that $Q_{Z}$ is even, positive definite, and unimodular. Observe that $L-\text{int}(K)$ is orientation-preserving diffeomorphic to $Z-\text{int}(Y)$ because $L-\text{int}(K)=i(C'-C)$ and $Z-\text{int}(Y)=C'-C$. This is illustrated in Figure 5. Now set $Z_n$ for each $n\in\mathbb{Z}^{> 0}$ equal to the result of simultaneously end-summing the standard $\mathbb{R}^4$ with $n$ copies of $Z$ (using an arbitrary choice of rays). For each $n\in\mathbb{Z}^{> 0}$, let $Y_n\subseteq Z_n$ (resp. $K_n\subseteq \natural_n L$) be the result of ambiently boundary summing together the copies of $Y$ (resp. $K$) coming from each copy of $Z$ (resp. $L$) used in the construction of $Z_n$ (resp. $\natural_n L$). It is clear that these ambient boundary sums can be performed so that each $Y_n\subseteq Z_n$ and $K_n\subseteq \natural_n L$  are the necessary pairs to ensure that property (\ref{end}) holds. Note also that each $Q_{Z_n}$ has rank $16sn$. 

We turn to property (\ref{embeddings}). This is the step that requires the first stage disk of $CH$ to have only negative double points, as we arranged for in the first paragraph. Notice that $W$ is the interior of a smoothly embedded $B^4\cup T_1\cup CH_1\cup\cdots\cup CH_g\subseteq X$, where $T_1$ is attaching to a $-1$-framed unknot in $\partial B^4$ and each $CH_i$ is a second stage Casson handle attaching to the frontier of $T_1$. Let $A\subseteq W$ be the result of smoothly isotoping $B^4\cup T_1$ into its own interior using the collared neighborhood of its boundary. Since $A$ is compact, we can arrange for $A$ to be contained in the interior of $C\subseteq W$ by choosing  $a\in\mathbb{R}^{>0}$ sufficiently large during the construction of $C$. This ensures that $L$ is a shaved $\mathbb{R}^4$ in $\overline{\mathbb{C}P^2}-i(A)$. We will now use relative Kirby calculus to see that $\overline{\mathbb{C}P^2}-i(A)$ is the interior of a compact Stein surface. Observe that Figure 6 provides a handle decomposition for $\overline{\mathbb{C}P^2}$ with $i(A)$ equal to the union of the $0$-handle, all $1$-handles, and the single $-1$-framed $2$-handle, so that $\overline{\mathbb{C}P^2}-i(A)$ is the interior of the relative handlebody $S$ described by  Figure 7.  Reversing the orientation on this relative handlebody, removing the interior of its $4$-handle, and then turning the result upside down (reversing its orientation again) produces a description of $S-(\text{open }4\text{-ball})$ as the relative handlebody in Figure 8.\linebreak We fill the new $S^3$ boundary with $B^4$ (in the only way possible) simply by blowing down the $4$-manifold described by bracketed handles along the $\left<1\right>$-framed $2$-handle in the final diagram of Figure 8. This results in a description of $S$ as the handlebody shown in Figure 9. Since the $0$-framed link in the final diagram of Figure 9 is Legendrian in $(S^3, \xi_\text{std})$ and each component has Thurston-Bennequin number equal to $1$, the criterion in \cite{G2} ensures that $S$ is indeed a compact Stein surface. Hence, $\overline{\mathbb{C}P^2}-i(A)$ is an open Stein surface that is the interior of the compact Stein surface $S$. Using our discussion after the definition of end-summing, we can now conclude that property (\ref{embeddings}) holds as well.  

\begin{figure}[t]
  \centering
  \begin{minipage}{.5\textwidth}
    \centering
    \includegraphics[scale=.7]{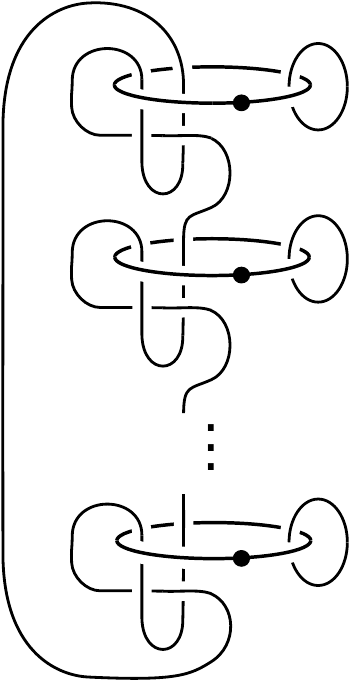}
    \caption{$\overline{\mathbb{C}P^2}$}
    \begin{picture}(0,0)
    \put(-44,165){$-1$}
    \put(34,159){$0$}
    \put(34,124){$0$}
    \put(34,66){$0$}
    \put(30,85){$\cup \;4$-handle}
    \end{picture}
  \end{minipage}%
  \begin{minipage}{.5\textwidth}
    \centering
    \includegraphics[scale=.7]{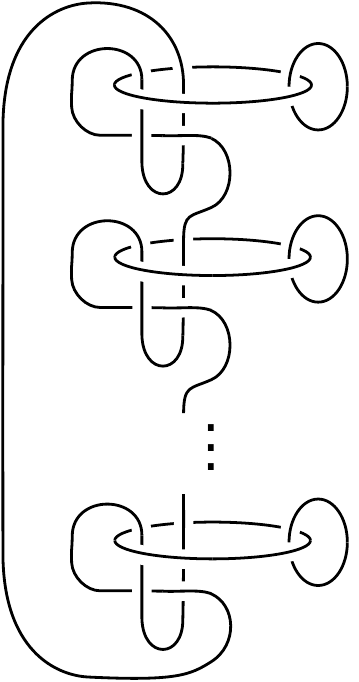}
    \caption{$S=\overline{\mathbb{C}P^2}-\text{int}(i(A))$}
    \begin{picture}(0,0)
    \put(-52,165){$\left<-1\right>$}
    \put(34,159){$0$}
    \put(34,124){$0$}
    \put(34,66){$0$}
    \put(30,85){$\cup \;4$-handle}
    \put(9,159.5){$\left<0\right>$}
    \put(9,125.5){$\left<0\right>$}
    \put(9,67.5){$\left<0\right>$}
    \end{picture}
  \end{minipage}
\end{figure}

\begin{figure}[t]
    \centering
    \vspace{.7cm}
    \includegraphics[scale=.7]{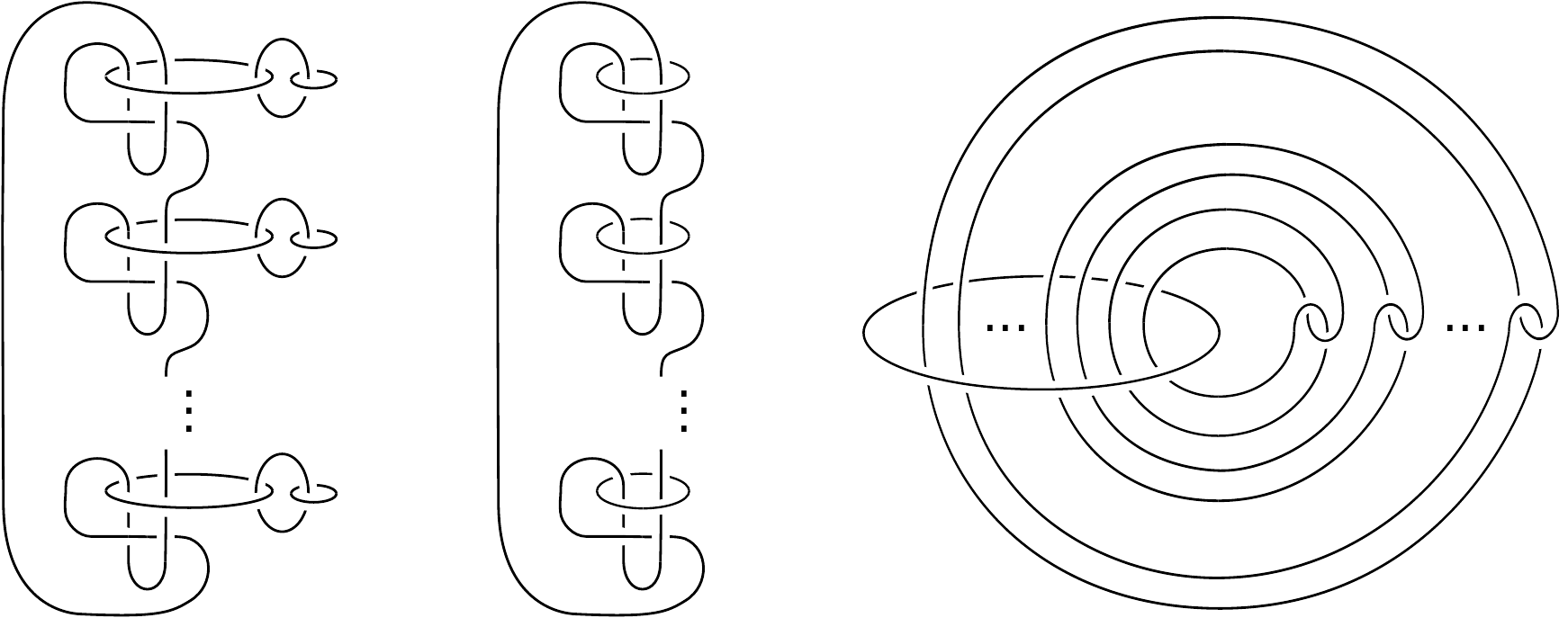}
    \caption{$S-$(open $4$-ball)}
    \begin{picture}(0,0)
    \put(-184,166){$\left<1\right>$}
    \put(-118,165){$\left<0\right>$}
    \put(-118,129){$\left<0\right>$}
    \put(-118,71){$\left<0\right>$}
    \put(-99,151){$0$}
    \put(-99,115){$0$}
    \put(-99,58){$0$}
    \put(-132,160){$\left<0\right>$}
    \put(-132,124){$\left<0\right>$}
    \put(-132,67.5){$\left<0\right>$}
    \put(-91, 90){\huge{$\sim$}}
    \put(-6, 90){\huge{$\sim$}}
    \put(-72,166){$\left<1\right>$}
    \put(-19,151){$0$}
    \put(-19,115){$0$}
    \put(-19,58){$0$}
    \put(14,108){$\left<1\right>$}
    \put(108,102){$0$}
    \put(134,126){$0$}
    \put(157,144){$0$}
    \end{picture}
\end{figure}

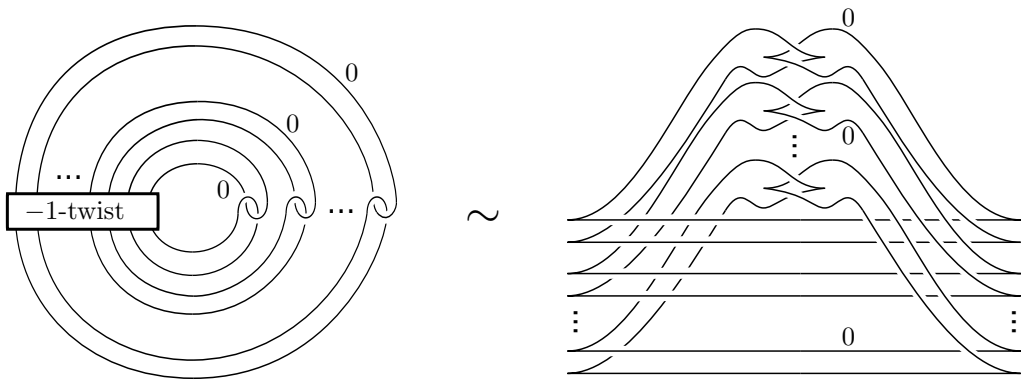
\begin{figure}[t]
    \centering
    \vspace{.7cm}
    \includegraphics[scale=.7]{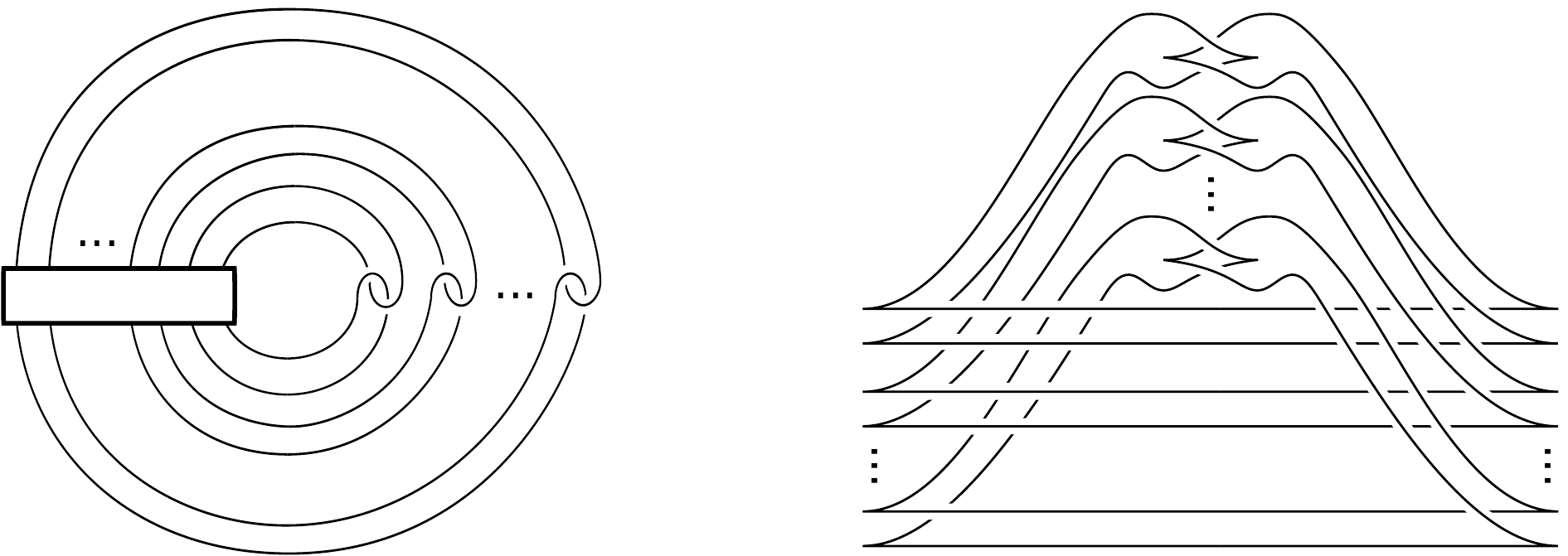}
    \caption{Filling with $B^4$ to obtain a description of $S$ as a handlebody.}
    \begin{picture}(0,0)
    \put(-19, 90){\huge{$\sim$}}    
    \put(-112,101){$0$}
    \put(-86,126){$0$}
    \put(-64,145){$0$}
    \put(-184, 93){$-1$-twist}
    \put(122, 166){$0$}
    \put(122, 121.5){$0$}
    \put(122, 45.5){$0$}
    \end{picture}
\end{figure}

We end with property (\ref{inv}). Since $S$ is a handlebody with no handles of index $1$, $3$ or $4$ and each of the $g$ $2$-handles in Figure 9 has even framing, the double of $S$ is diffeomorphic to $\# g S^2\times S^2$. So ${L}$ admits a smooth, shaved, orientation-preserving embedding into $\# g S^2\times S^2$. Then our discussion after the definition of end-summing can again be applied to conclude that each $\natural_n L$ for $n\in\mathbb{Z}^{> 0}$ admits a smooth embedding into $\# (ng) S^2\times S^2$. In particular, $\gamma(\natural_n {L})\leq ng <\infty$ for each $n\in\mathbb{Z}^{> 0}$. On the other hand, it follows from another standard standard cut-and-paste argument (e.g. see \cite{BE} or \cite{Tay}) using property (\ref{end}), the rank of each $H_2(Z_n)$ that was computed above, and Furuta's 10/8 Inequality (see Theorem 1 of \cite{Fu}) that $\gamma(\natural_n{L})>2ns$. In particular, choosing $s$ to be sufficiently large ensures that $\gamma(L)$ is an  arbitrarily large integer. Also, $\lim_{n\to\infty} \gamma(\natural_n{L})=\infty$. So each claim in property (\ref{inv}) holds as well.   
\end{proof}

The following corollary finds $\mathbb{R}^4$'s that display the opposite behavior from those produced in \cite{G4}. Specifically, the $\mathbb{R}^4$'s from \cite{G4} form an uncountable collection contained in $\mathbb{C}^2$ that each inherit a Stein structure. 
\begin{cor}
\label{notStein}
There is a compact Stein surface $S$ that contains uncountably many diffeomorphism types of exotic $\mathbb{R}^4$'s that each fail to admit a Stein structure.  Furthermore, these $\mathbb{R}^4$'s all realize  the same value of the Taylor invariant and $S$ can be chosen so that this value is an arbitrarily large integer.  
\end{cor}

\begin{proof}
Let $L$ be an $\mathbb{R}^4$ constructed by Theorem \ref{inStein} to realize an arbitrarily large (finite) value of the Taylor invariant. Choose $S$ to be the corresponding compact Stein surface. By a standard argument  (e.g. see \cite{Tay}) that applies Taubes's extension of Donaldson theory in \cite{Tau}, we can use property (\ref{end})\linebreak above to conclude that there are uncountably many diffeomorphism types of shaved $\mathbb{R}^4$'s in $L$ that each have Taylor invariant equal to the Taylor invariant of $L$. Each of these $\mathbb{R}^4$'s admits a smooth, shaved, orientation-preserving embedding into $S$. As noted earlier, Remark 4.5 of \cite{Tay} ensures that none of these $\mathbb{R}^4$'s admit a Stein structure. 
\end{proof}
 \noindent The exotic $\mathbb{R}^4$'s produced by this corollary are actually distinguishable by their compact equivalence classes. Alternatively, it is possible to modify this construction to arrange for these $\mathbb{R}^4$'s to all realize the same compact equivalence class but still represent distinct diffeomorphism classes. We will define compact equivalence classes and discuss this in more detail in the next section. 
 
Next, we produce $\mathbb{R}^4$'s that behave similarly on the level of the Taylor invariant but do not interact nicely with Stein surfaces or definite $4$-manifolds. These are defined from an exhaustion of the exotic $\mathbb{R}^4$, denoted by $U$, that was introduced by Freedman and Taylor in \cite{FT}. Finding exotic $\mathbb{R}^4$'s by exhausting $U$ was first done in Corollary D of \cite{FT} and was subsequently studied more carefully in \cite{G1}, but neither considered embeddings into Stein surfaces. We remark that $U$ is typically referred to as the universal $\mathbb{R}^4$ because it is characterized as the unique $\mathbb{R}^4$ that contains every $\mathbb{R}^4$ as an end-summand.

\begin{lemma}
\label{smalllemma}
Suppose that $K$ is a topologically slice knot in $S^3$ that has a Legendrian representative in $(S^3, \xi_\text{std})$ with Thurston-Bennequin number equal to $1$. Let $\overline{K}$ be the knot in $S^3$ obtained from $K$ by switching all crossings, isotoped so that the two knots are disjoint and not linking. Then the handlebody $Y_K$ defined by attaching $0$-framed $2$-handles to $B^4$ along both $K$ and $\overline{K}$ smoothly embeds into $U$ but does not smoothly embed into any Stein surface (either open or compact). In particular, this holds whenever $K$ is the positive (untwisted) Whitehead double of a Legendrian knot in $(S^3, \xi_\text{std})$ with non-negative Thurston-Bennequin number. 
\end{lemma} 

\begin{proof} 
It easily follows from Lemma 1.1 of \cite{G1} that $Y_K$ admits a smooth embedding into some $\mathbb{R}^4$, so that the universal property of $U$ guarantees a smooth embedding of $Y_K$ into $U$. To obtain a contradiction, suppose that $Y_K$ also admits a smooth embedding into some Stein surface $S$. By passing to the interior of the appropriate component, we can assume that $S$ is both connected and open. Construct handlebodies $X_K$ and $X_{\overline{K}}$ by attaching a $0$-framed $2$-handle to $B^4$ along $K$ and $\overline{K}$, respectively. Notice boundary summing $X_K$ and $X_{\overline{K}}$ produces $Y_K$.  Our requirement on the Thurston-Bennequin number of $K$ ensures that the interior $S'$ of $X_K$ is also an open Stein surface, so that the result $S''$ of end-summing $S$ and $S'$ (using any choice of rays) is again an open Stein surface by \cite{CG}. Since $\overline{X_K}=X_{\overline{K}}$, restricting the embedding of $Y_K$ into $S$ to either $X_K$ or $X_{\overline{K}}$ (depending on weather it is orientation-preserving or orientation-reversing) provides a smooth, orientation-preserving embedding of $X_{\overline{K}}$ into the copy of $S$ in $S''$. Isotoping $X_K$ into its own interior using the collared neighborhood of its boundary defines a smooth, orientation-preserving embedding of $X_K$ into the copy of $S'$ in $S''$. Using these two embeddings, ambiently boundary sum  $X_{\overline{K}}\subseteq S\subseteq S''$ and $X_K\subseteq S'\subseteq S''$ to produce a new smoothly embedded copy of $Y_K$ in $S''$. After choosing the appropriate orientations on $K$ and $\overline{K}$, the knot $K\#\overline{K}$ is smoothly slice. The corresponding slice disk ensures the existence of a smoothly embedded sphere in $Y_K$ representing the sum of a generating pair for $H_2(Y_K)=\mathbb{Z}\oplus \mathbb{Z}$. This sphere includes under the new embedding of $Y_K$ to a smoothly embedded, homologically essential sphere with trivial square in $S''$.  However, it is well-known that this cannot happen in a Stein surface. So we have reached the necessary contradiction. The final claim from the statement follows from the proof of Theorem 3.4 in \cite{AM} and the well-known fact that all Whitehead doubles are topologically slice. 
\end{proof} 

\begin{theorem}
\label{notinStein}
There are exotic $\mathbb{R}^4$'s realizing arbitrarily large (finite) values of the Taylor invariant that each contain a smooth, compact, codimension-$0$ submanifold that does not smoothly embed into any Stein surface  (either open or compact) or any smooth, closed, definite $4$-manifold. These compact submanifolds each admit a smooth embedding into $U$.
\end{theorem}

\begin{proof}
It is immediate from Lemma \ref{smalllemma} that there there is a smooth, compact, codimension-$0$ submanifold $C_1$ in $U$ that does not smoothly embed into any Stein surface. (Alternatively, it is possible to prove the existence of $C_1$ directly by returning to the definition of $U$ from \cite{FT}.)  The universal property of $U$ guarantees that there is also a smooth, compact, codimension-$0$ submanifold $C_2$ in $U$ that does not smoothly embed into any smooth, closed, definite $4$-manifold, as there are several known constructions of exotic $\mathbb{R}^4$'s that have a submanifold satisfying these requirements (e.g. end-sum one of the $\mathbb{R}^4$'s constructed in \cite{G1} with itself using two different orientations). To define the necessary $\mathbb{R}^4$'s, fix an exhaustion $E_0\subseteq E_1\subseteq E_2\subseteq \cdots$ of $U$ with each $E_i$ equal to the interior of a topological $4$-ball in $U$. Choose this exhaustion so that each $E_i$ contains both $C_1$ and $C_2$. Then each $E_i$ is an $\mathbb{R}^4$ (inheriting its smooth structure from $U$) that has the necessary compact submanifold. Observe also that each $\gamma(E_i)<\infty$ because each $E_i$ admits a smooth embedding into a smooth, closed, spin $4$-manifold with hyperbolic intersection form, obtained by doubling a smooth, compact, codimension-$0$ submanifold of $U$. It is clear that $\lim_{i\to\infty} \gamma(E_i)=\gamma(U)=\infty$. 
\end{proof}

Comparing the previous theorem to Gompf's results from \cite{G5} can sometimes produce two non-diffeomorphic smooth structures on the same underlying topological $4$-manifold, distinguishable by the existence of a smooth, compact, codimension-$0$ submanifold that does not smoothly embed into a Stein surface. In particular, we generalize  the first statement of Theorem 3.4 from \cite{G5} without using either the genus function or the Taylor invariant, which were both necessary to the proof given in that paper. 
\begin{cor} 
\label{usingU}
Every open, topological $4$-manifold can be smoothed so that it contains a smooth, compact, codimension-$0$ submanifold that does not smoothly embed into any Stein surface. As a result, every open, topological $4$-manifold $X$ that topologically embeds into a $4$-dimensional handlebody with all indices $\leq 2$ admits at least two diffeomorphism classes of smooth structures. In particular, the interior of any $4$-dimensional handlebody with all indices $\leq 2$ admits at least two diffeomorphism classes of smooth structures.  
\end{cor}

\begin{proof}
Suppose that $X$ is an open, topological $4$-manifold.  Section 8.2 of \cite{FQ} produces a smooth structure $\Sigma$ on $X$.  It follows from Theorem \ref{notinStein} that $U$ contains a smooth, compact, codimension-$0$ submanifold that does not smoothly embed into any Stein surface.  If $X$ is orientable, the necessary smooth structure $\Sigma'$ on $X$ is obtained by end-summing $X_\Sigma$ with $U$ (using any orientations and any choice of rays). If $X$ is nonorientable,  fixing an orientation on a tubular neighborhood of some smooth, properly-embedded ray $\gamma$ in $X_\Sigma$ allows us to still perform this end-sum operation to obtain the necessary smooth structure $\Sigma'$ on $X$. So the first claim holds.   

Next, suppose that $X$ topologically embeds into a  $4$-dimensional handlebody $H$ with all indices $\leq 2$. If $X$ is orientable, set $\tilde{X}=X$. Otherwise, set $\tilde{X}$ equal to the orientable double cover of $X$. Choose $\tilde{H}$ analogously for $H$ and let $\pi:\tilde{H}\to H$ denote the corresponding covering map (equal to the identity in the orientable case).  So $X$ inherits a smooth structure $\Sigma''$ from a Casson smoothing on the interior of $H$ produced by Lemma 3.2 of \cite{G5}, with this Casson smoothing  chosen so that its lift under $\pi$ smoothly embeds into some Stein surface. Notice that  $\tilde{X}$ topologically embeds into $\tilde{H}$ and $\pi$ restricts to a covering map $\pi_X:\tilde{X}\to X$. Then $\Sigma''$ lifts to a smooth structure $\tilde{\Sigma}''=\pi_X^*\Sigma''$ on $\tilde{X}$ and $\tilde{X}_{\tilde{\Sigma}''}$ smoothly embeds into this Stein surface. For $\tilde\Sigma'=\pi_X^*\Sigma'$, it is clear that $\tilde X_{\tilde\Sigma'}$ and $\tilde X_{\tilde\Sigma''}$ are not diffeomorphic. As needed, this means that $X_{\Sigma'}$ and $X_{\Sigma''}$ are also not diffeomorphic. The final sentence of the statement clearly follows as well.   
\end{proof}

Our final corollary of this section uses Theorem \ref{inStein} to produce an exotic $\mathbb{R}^4$ with infinite Taylor invariant that smoothly embeds into an open Stein surface, and then applies Theorem \ref{notinStein} to conclude that this property distinguishes it from  $U$.  The results from Section 7 of \cite{G5} allow us to subsequently locate uncountably many $\mathbb{R}^4$'s meeting this description. This technique is discussed in more detailed at the end of the next section. Combining all of the known work about exotic $\mathbb{R}^4$'s actually produces many examples with infinite Taylor invariant, some of which are also not diffeomorphic to $U$. For example, it is straightforward to modify Example 5.10 of \cite{Tay} to find uncountably many $\mathbb{R}^4$'s in ${\mathbb{C}P^2}$ with infinite Taylor invariant. However, it is again unclear how any of these previous constructions relate to Stein surfaces. 
\begin{cor}
\label{notU}
There is an exotic $\mathbb{R}^4$, denoted ${L}_\infty$, with infinite Taylor invariant that admits a smooth embedding into an open Stein surface. Additionally, every smooth, compact, codimension-$0$ submanifold of $L_\infty$ admits a smooth embedding into a finite connected sum of $\overline{\mathbb{C}P^2}$'s. Either of these properties is sufficient to distinguish the diffeomorphism type of $L_\infty$ from the diffeomorphism type of $U$. Furthermore, $L_\infty$ contains uncountably many diffeomorphism types of exotic $\mathbb{R}^4$'s that also have infinite Taylor invariant.
\end{cor}

\begin{proof}
Let $L_\infty=\natural_\infty L$, where $L$ is any exotic $\mathbb{R}^4$ produced by Theorem \ref{inStein}. It is clear from\linebreak properties (1), (2) and (4) of Theorem \ref{inStein} that $\gamma(L_\infty)=\infty$ and the necessary embeddings exist. On the other hand,  Theorem \ref{notinStein} guarantees that $U$ does not satisfy either of these properties. Finally, Theorem 7.1 of \cite{G5} locates the necessary $\mathbb{R}^4$'s inside of $L_\infty$ to ensure the final claim holds as well. 
\end{proof}

Determining when constructions of $\mathbb{R}^4$'s with infinite Taylor invariant provide a new description of $U$ seems to be a difficult problem in general. For example, equipping $L_\infty$ with the orientation induced from the orientation on $L$ allows us to define $L_\infty\natural \overline{L_\infty}$. Notice that $L_\infty\subseteq L_\infty\natural \overline{L_\infty}\subseteq U$. It mimics the behavior of $U$ in the following diffeomorphisms:  $L_\infty\natural (L_\infty\natural \overline{L_\infty})\cong_\text{diff} L_\infty\natural \overline{L_\infty}\cong_\text{diff} (L_\infty\natural \overline{L_\infty}) \natural \overline{L_\infty}$. \linebreak This new exotic $\mathbb{R}^4$ also shares the property with $U$ that it contains a smooth, compact, codimension-$0$ submanifold that does not smoothly embed into any smooth, closed, definite $4$-manifold (e.g. apply the argument from Theorem 9.4.3 of \cite{GS} to copies of $L$ and $\overline{L}$ in $L_\infty \natural \overline{L_\infty}$). In particular, $L_\infty\natural \overline{L_\infty}$ is certainly not diffeomorphic to $L_\infty$. On the other hand, it is possible that $L_\infty\natural \overline{L_\infty}$ still shares the property with $L_\infty$ that it admits a smooth embedding into some Stein surface.

\section{The Taylor invariant, the genus-rank function, and compact equivalence classes}

To illustrate the usefulness of the $\mathbb{R}^4$'s defined in the previous section, we consider their applications to constructing and manipulating smooth structures on open $4$-manifolds. In particular, we show that the genus function can be preserved while end-summing with ``arbitrarily large" $\mathbb{R}^4$'s, or at least only disrupted it in a small, controlled way. Conversely, we also see that this statement is false for the universal $\mathbb{R}^4$. Our observations will facilitate a procedure for independently controlling the Taylor invariant, the genus-rank function, and compact equivalence classes, often realizing infinitely many values of one invariant while keeping the other two fixed. 

We start with definitions of the genus and genus-rank functions, following those given in \cite{G5}.  
\begin{definition}
Let $X$ be a smooth, oriented $4$-manifold with torsion subgroup $T\subseteq H_2(X)$. The \textit{genus function} $G:H_2(X)\to\mathbb{Z}^{\geq 0}$ is the function assigning to each $\alpha\in H_2(X)$ the smallest possible genus of a smoothly embedded, compact, oriented surface $F$ representing $\alpha$. If $H_2(X)/T$ is a free abelian group, then the corresponding \textit{genus-rank function} is the function $\mathbb{Z}^{\geq 0}\to\mathbb{Z}^{\geq 0}\cup\{\infty\}$ that sends each $g\in\mathbb{Z}^{\geq 0}$ to the rank of the rational span in $H_2(X)/T$ of all $\alpha\in H_2(X)$ with $G(\alpha), |\alpha\cdot \alpha|\leq g$. 
\end{definition}
\noindent Notice that the genus-rank function is defined whenever $X$ is homeomorphic to the interior of a handlebody with only finitely many $3$-handles. The genus-rank function is an invariant of the diffeomorphism type of $X$, and changing the smooth structure on $X$ can effect this invariant. In particular, we will frequently be able to distinguish smooth structures on the same underlying topological $4$-manifold by the smallest $g\in\mathbb{Z}^{\geq 0}$ on which their genus-rank functions have a nonzero value. We call this integer the \textit{first characteristic genus} of the corresponding genus-rank function. (As in \cite{G5}, characteristic genera can also be defined more generally.) Although the genus-rank function is sometimes denoted by $\gamma$, we will reserve that symbol for the Taylor invariant.

 Stein surfaces come equipped with an adjunction inequality that provides lower bounds on the genus function (e.g. see Section 2 of \cite{G5}). If $X$ is a smooth, oriented, open $4$-manifold that admits a smooth, orientation-preserving embedding $i:X \hookrightarrow S$ for some Stein surface $S$ (oriented by its complex structure), then it inherits the adjunction inequality associated to $S$. More specifically, every $\alpha\in H_2(X)$ with $i_*(\alpha)\neq 0$ must satisfy
\[2G(\alpha)-2\geq \alpha\cdot\alpha+|\left<c_1(S), i_*(\alpha)\right>|.\]
It is precisely this observation that is used in \cite{G5} to produce point-wise upper bounds on the genus-rank function, and to subsequently distinguish large families of smooth structures on handlebody interiors. 

Our first task is to consider how end-summing with exotic $\mathbb{R}^4$'s effects the genus-rank function. Recall from the previous section that simultaneously end-summing a smooth, oriented, open $4$-manifold with a collection of oriented $\mathbb{R}^4$ defines a new smooth structure on the same underlying topological $4$-manifold. As above, the resulting smooth $4$-manifold is oriented by the orientation on this underlying topological $4$-manifold. The new genus-rank function is point-wise bounded below by the genus-rank function realized by the original smooth structure. The challenge seems to be finding upper bounds. To address this issue, the following lemma allows us to end-sum with certain large $\mathbb{R}^4$'s while preserving any adjunction inequalities that are inherited from embeddings into Stein surfaces. 

\begin{lemma}
\label{almoststein}
Suppose that $L$ is an oriented $\mathbb{R}^4$ constructed by Theorem \ref{inStein}. If $X_{\Sigma'}$  is a smooth, oriented $4$-manifold produced by simultaneously end-summing a smooth, oriented, open $4$-manifold $X_\Sigma$ with each oriented $\mathbb{R}^4$ from some (possibly infinite) collection $\{\natural_{n_i}{L} \mid n_i\in \mathbb{Z}^{> 0}\cup\{\infty\}\}_{i=1}^N$, then the genus function on $X_{\Sigma'}$ satisfies every adjunction inequality inherited from a smooth, orientation-preserving embedding of $X_\Sigma$ into an open Stein surface that sends each ray used in the end-sum operation to a smooth, properly embedded ray in the Stein surface.
\end{lemma}

\begin{proof}
 We proceed by induction on $N$. First consider the case when $N=1$, so that the smooth structure ${\Sigma'}$ on $X$ is obtained by end-summing $X_\Sigma$ with $\natural_{n} L$ for $n=n_1\in \mathbb{Z}^{> 0}\cup\{\infty\}$. Let $\gamma$ be the ray in $X_\Sigma$ used to perform this end-sum operation. Suppose that $i:X_\Sigma\hookrightarrow S$ is a smooth, orientation-preserving embedding of $X_\Sigma$ into an open Stein surface $S$ that sends $\gamma$ to a smooth, properly-embedded ray $\gamma'$ in $S$. Property (\ref{embeddings}) of Theorem \ref{inStein} ensures that $\natural_{n}{L}$ sits inside of an open, connected Stein surface $S_n$ as a shaved $\mathbb{R}^4$. Let $S'$ be obtained by end-summing $S$ and $S_n$ along the ray $\gamma'$ in $S$ and an arbitrary ray in $S_n$. Then $S'$ admits a Stein structure by \cite{CG}. Using the fact that $\natural_nL$ is a shaved $\mathbb{R}^4$ in $S_n$, we can find a smooth embedding $p:[0,1]\to S'$ whose intersection with $S\subseteq S'$ is the ray $\gamma'$ and whose intersection with $\natural_nL\subseteq S_n\subseteq S'$ is some smooth, properly embedded ray in $\natural_n L$. The union of $i(X_\Sigma)\subseteq S\subseteq S'$, $\natural_kL\subseteq S_n\subseteq S'$, and an open regular neighborhood of this path $p$ is a smoothly embedded copy of $X_{\Sigma'}$ in $S'$. The Stein structure on $S'$ can be chosen so that this embedding is orientation-preserving and so that $c_1(S')=c_1(S)\oplus c_1(S_n)\in H^2(S')=H^2(S)\oplus H^2(S_n)$. Thus, we've constructed a smooth, orientation-preserving embedding $i':X_{\Sigma'}\hookrightarrow S'$ that is an extension of the embedding $i:X_\Sigma\hookrightarrow S$. For any $\alpha\in H_2(X)$, observe that $H_2(S')=H_2(S)\oplus H_2(S_n)$ and 
\[\left<c_1(S'), i_*'(\alpha)\right>=\left<c_1(S)\oplus c_1(S_n), i_*(\alpha)\oplus 0\right>=\left<c_1(S), i_*(\alpha)\right>.\]
Also, $i_*(\alpha)=0$ if and only if $i_*'(\alpha)=0$.  It follows that the adjunction inequality on $X_\Sigma$ inherited from the embedding $i:X_\Sigma\hookrightarrow S$ is also inherited by $X_{\Sigma'}$ from the embedding $i':X_{\Sigma'}\hookrightarrow S'$. So the statement holds when $N=1$. Furthermore, if $i$ sends a collection of smooth, properly embedded rays in $X_\Sigma$ to smooth, properly embedded rays in $S$, then this collection can be smoothly isotoped in $X_\Sigma$ so that $i'$ sends each to a smooth, properly embedded ray in $S'$. Now induction provides the necessary result when $N$ is finite. If $N=\infty$, then the result follows because every smoothly embedded, compact, oriented surface in $X_{\Sigma'}$ is contained in some $X_{\Sigma''}$ sitting inside $X_{\Sigma'}$ that is obtained from $X_\Sigma$ by performing end-sums with only finitely many $\natural_{n_i} {L}$. 
\end{proof}

As an immediate corollary, we can preserve the adjunction inequality on Stein surfaces while end-summing with $\mathbb{R}^4$'s that realize arbitrarily large values of the Taylor invariant.
\begin{cor}
\label{stein}
There is a large, oriented $\mathbb{R}^4$, denoted by ${L}$, with the property that any smooth, oriented $4$-manifold $X_{\Sigma'}$ produced by simultaneously end-summing an open Stein surface $X_\Sigma$ with each $\mathbb{R}^4$ from some (possibly infinite) collection $\{\natural_{n_i}{L} \mid n_i\in \mathbb{Z}^{> 0}\cup\{\infty\}\}_{i=1}^N$ satisfies the adjunction inequality associated to this Stein surface.  Furthermore, ${L}$ can be chosen so that $\gamma(L)$ is an arbitrarily large integer, each $\gamma(\natural_n{L})<\infty$ for $n\in \mathbb{Z}^{\geq 0}$, and $\lim_{n\to \infty} \gamma(\natural_n {L})=\infty$. If $\gamma(L)\neq 0$, then $L$ does not admit a Stein structure. 
\end{cor}
\begin{proof}
The first claim follows from Lemma \ref{almoststein} because the identity on $X_\Sigma$ obviously sends every smooth, properly embedded ray to a smooth, properly embedded ray. Property (1) of Theorem \ref{inStein} guarantees the claims about the Taylor invariant of $L$. As above, Remark 4.5 of \cite{Tay} ensures that $L$ does not admit a Stein structure if $\gamma(L)\neq 0$. 
\end{proof}

The behavior of the exotic $\mathbb{R}^4$'s in the previous lemma is not what we should expect in general. Our next example illustrates the different ways that end-summing with large $\mathbb{R}^4$'s can effect the genus function.  
\begin{example}
\label{endsummingex}
Let $X$ be obtained by connect summing infinitely many copies of $S^2\times S^2$ and let $\Sigma_\text{std}$ denote the standard smooth structure on $X$. Arrange for $X$ to have a unique end by performing this sum in the appropriate way. End-summing $X_{\Sigma_\text{std}}$ with any exotic $\mathbb{R}^4$ will always produce the standard smooth structure again (e.g. see Remark 6.5 of \cite{Tay}), but Example 3.6 of \cite{G5} constructs uncountably many diffeomorphism classes of Casson smoothings on $X$. These diffeomorphism classes are distinguished by the corresponding genus-rank functions (or their relative versions). Any smooth $4$-manifold obtained by equipping $X$ with one of these Casson smoothings admits a smooth, orientation-preserving embedding into a Stein surface, and the inherited adjunction inequality ensures that it does not contain a smoothly embedded, homologically essential sphere with trivial square. Therefore, Lemma \ref{almoststein} allows us to repeatedly end-sum  each of these new smooth $4$-manifolds with large $\mathbb{R}^4$'s and never recover the standard smooth structure on $X$. These $\mathbb{R}^4$'s can be chosen to have arbitrarily large (finite) values of the Taylor invariant or to have infinite Taylor invariant. With a bit more care, we can arrange for each Casson smoothing to be Stein, each element of the standard basis for $H_2(X)$ to realize equality in the corresponding adjunction inequalities, and one Casson handle from infinitely many $S^2\times S^2$ summand in each smoothing to be some fixed Casson handle $CH_0$. Now end-summing with these exotic $\mathbb{R}^4$'s actually preserves the value of each genus function on this basis.  On the other hand, consider the $\mathbb{R}^4$'s produced by Theorem \ref{notinStein}. These each have finite Taylor invariant but form an exhaustion of $U$, and it follows from the construction of $U$ in \cite{FT} that we can choose some $E$ from this exhaustion so that end-summing any of these Stein-Casson smoothings on $X$ with $E$ allows any one of the $CH_0$'s to be  isotoped to a standard $2$-handle. This defines at least one standard copy of $S^2\times S^2 -\{\text{pt}\}$. End-summing instead with $\natural_n E$ for any $n\in\mathbb{Z}^{> 0}\cup\{\infty\}$, which still has finite Taylor invariant when $n$ is finite, produces smooth $4$-manifolds that each have at least $n$ standard copies of $S^2\times S^2-\{\text{pt}\}$. When $n=\infty$, we can even conclude that each resulting smooth structure lies in the same isotopy class as $\Sigma_{\text{std}}$ (e.g. apply the proof of Theorem 3 in \cite{FT}). Similarly, end-summing any of these Stein-Casson smoothings with $U$ also recovers the standard smooth structure. In reference to the discussion at the end of the previous section, it follows that $\natural_\infty E$ might be another good candidate to provide an alternate description of $U$. \end{example}

Next, we introduce compact equivalence classes and some relevant terminology. Like the Taylor invariant, compact equivalence classes are designed to measure the complexity of shaved $\mathbb{R}^4$'s in smooth $4$-manifolds. Our definition extends the one given in \cite{G1} to smooth $4$-manifolds that are not necessarily homeomorphic to $\mathbb{R}^4$. 
\begin{definition}
Let $\Sigma_1$ and $\Sigma_2$ be smooth structures on an oriented, topological $4$-manifold $X$.  We write $\Sigma_1\leq \Sigma_2$ if every flat, topological embedding $e_1:B^4\hookrightarrow X_{\Sigma_1}$ that is smooth around some $p_1\in\partial B^4$ corresponds to a flat, topological embedding $e_2:B^4\hookrightarrow X_{\Sigma_2}$ that is smooth around some $p_2\in \partial B^4$ with the property that interiors of $e_1(B^4)$ and $e_2(B^4)$ are orientation-preserving diffeomorphic. If $\Sigma_1\leq \Sigma_2\leq \Sigma_1$, then we say that $\Sigma_1$ and $\Sigma_2$ are compactly equivalent. 
\end{definition}
\noindent The relation $\leq$ descends to a partial ordering on compact equivalence classes. Similarly, the Taylor invariant is well-defined on compact equivalence classes on spin $4$-manifolds. More precisely, the Taylor invariant is monotonic with respect to the ordering defined by $\leq$ if the underlying topological $4$-manifold is spin. However, we will see that same value of the Taylor invariant can often be realized by distinct compact equivalence classes. 

Typical methods for changing the compact equivalence class of a given smooth structure require it to come with some embeddings into smooth, closed, definite $4$-manifolds. Our next definition introduces one of the more general version of this condition. 
\begin{definition}
Let $\Sigma$ be a smooth structure on an oriented, topological $4$-manifold $X$. We say that $\Sigma$ is \textit{compactly positive definite} (resp. \textit{compactly negative definite}) if every $\Sigma$-smooth, compact, codimension-$0$ submanifold of $X$ admits a $\Sigma$-smooth, orientation-preserving embedding into a smooth, closed, simply connected, positive definite (resp. negative definite) $4$-manifold. 
\end{definition}
\noindent On open, topological $4$-manifolds admitting a smooth structure that is either compactly negative definite or compactly positive definite, there are at least two essentially different ways to end-sum with exotic $\mathbb{R}^4$'s to produce uncountably many new diffeomorphism types. As we will see below, these two methods can be differentiated by whether or not the compact equivalence class is the same for all resulting smooth structures. 

In our first application of Lemma \ref{almoststein}, we end-sum with exotic $\mathbb{R}^4$'s to manipulate the Taylor invariant while maintaining control over the genus-rank function. Conversely, we also use methods from \cite{G5} to modify the genus-rank function without changing the Taylor invariant.  
\begin{theorem}
\label{TG}
Suppose that $X$ is a open, connected, topological $4$-manifold that is the interior of an oriented handlebody $H$ with all indices $\leq 2$, $0<\beta_2(X)<\infty$, and $w_2(X)\in H^2(X, \mathbb{Z}_2)$ is compactly supported. Then there are smooth structures on $X$ that realize infinitely many (arbitrarily large and finite) values of Taylor invariant but all produce the same genus-rank function, and infinitely many genus-rank functions occur in this way. If $X$ is spin, then there are also smooth structure on $X$ that produce infinitely many genus-rank functions but all realize the same arbitrarily large (finite) value of the Taylor invariant.
\end{theorem} 
\noindent Thus, we can independently control the Taylor invariant and genus-rank function under sufficiently nice conditions. Using techniques from the proof of Lemma 7.3 in \cite{G5}, we can sometimes realize each of these pairs by uncountably many distinct compact equivalence classes:
\begin{addendum}
\label{ad}
If the standard smooth structure that $X$ inherits as a handlebody interior is compactly positive definite, then each pair of Taylor invariant and genus-rank function from this theorem is realized by smooth structures on $X$ whose compact equivalence classes have the order type of $\mathbb{R}^{\geq 0}$. 
\end{addendum}

\begin{proof}[Proof of Theorem \ref{TG} and Addendum \ref{ad}]
These smooth structures will be constructed by end-\linebreak summing Casson smoothings on $X$ from \cite{G5} with $\mathbb{R}^4$'s from Theorem \ref{inStein}. In order to do this while sometimes realizing uncountably many compact equivalence classes, we must first consider these $\mathbb{R}^4$'s more carefully. Fix some $L$ from Theorem \ref{inStein}, a positive real number $t_0$, and a homeomorphism $h:\mathbb{R}^4\to {L}$. Apply Quinn's Stable Homeomorphism Theorem (Theorem 8.1A of \cite{FQ}) to isotope $h$ so that it is smooth on a neighborhood of the positive real line in $ \mathbb{R}^4$. For each $t\in\mathbb{R}^{\geq 0}$, let $B_t\subseteq \mathbb{R}^4$  denote the standard ball of radius $t_0+t$ and let ${L}_t$ denote the interior of $h(B_t)$. We can chose $t_0$ large enough that $K_1\subset {L}_{t_0}\subset {L}$, where $K_1$ is the compact set from property (\ref{end}) of Theorem \ref{inStein}. It is easy to conclude that each $L_t$ is an oriented $\mathbb{R}^4$ that also satisfies properties (1)-(4) of Theorem \ref{inStein}.  In particular, Lemma \ref{almoststein} holds with ${L}$ replaced by any ${{L}}_t$ because it only required property (2) from this theorem. We can now freely apply the properties ensured by Theorem \ref{inStein} and the conclusion of Lemma \ref{almoststein} to $L_t$ throughout this proof. 
 
We proceed to define the necessary smooth structures on $X$. First, equip $X$ with the orientation it inherits as a handlebody interior. Apply the hypothesis that $\beta_2(X)\neq 0$ to construct a collection $\{\Sigma_a\mid a\in\mathbb{Z}^{> 0}\}$ of Casson smoothings on $X$ corresponding to the handle decomposition from $H$ using the procedure from the proof of Theorem 3.4 of \cite{G5}. This procedure inductively constructs each $\Sigma_a$ so that $X_{\Sigma_a}$ inherits enough adjunction inequalities to ensure that the first characteristic genus of its genus-rank function is strictly larger than the first characteristic genus of the genus-rank function on any $X_{\Sigma_{a'}}$ with $a'<a$. Let $\Sigma_{\text{std}}$ denote the standard smooth structure that $X$ inherits a handlebody interior and fix a smooth, properly embedded ray $\gamma$ in $X_{\Sigma_{\text{std}}}$ that is contained in the interior of some $0$-handle from $H$. Also fix the corresponding smooth, properly embedded ray $\gamma_a$ in each $X_{\Sigma_a}$.  For each $a, n\in\mathbb{Z}^{> 0}$ and $t\in\mathbb{R}^{\geq 0}$, define $\Sigma_{a, n, t}$ to be the smooth structure on $X$ obtained by end-summing $X_{\Sigma_a}$ with $\natural_n {L}_t$ along the ray $\gamma_a$ in $X_{\Sigma_a}$ and an arbitrary ray in $\natural_n L_t$. It now suffices to verify that these smooth structures realize the necessary invariants. 

We start by considering the genus-rank functions that are produced by these new smoothings. For each $a\in\mathbb{Z}^{> 0}$, let $g_a\in\mathbb{Z}^{\geq 0}$ be the first characteristic genus of the genus-rank function on $X_{\Sigma_a}$. As described above, any $\Sigma_a$ and $\Sigma_{a'}$ for $a>a'$ can be distinguished because $X_{\Sigma_a}$ inherits enough adjunction inequalities from smooth, orientation-preserving embeddings into Stein-Casson smoothings on $X$ to ensure that $g_{a}>g_{a'}$. Each of these embeddings for a given $X_{\Sigma_a}$ restricts to the identity away from its Casson handles, so that each sends $\gamma_a$ to some smooth, properly-embedded ray in the corresponding Stein-Casson smoothing on $X$. Hence, Lemma \ref{almoststein} ensures that each $X_{\Sigma_{a, n, t}}$ inherits these same adjunction inequalities as $X_{\Sigma_a}$. So the genus-rank function on each $X_{\Sigma_{a, n, t}}$ sends $g_{a'}$ to zero if $a'<a$. On the other hand, the genus-rank function on each $X_{\Sigma_{a, n, t}}$ is point-wise bounded below by the genus-rank function on $X_{\Sigma_a}$. In particular, the genus-rank function on each $X_{\Sigma_{a, n, t}}$ sends $g_a$ to a nonzero integer.  Therefore, any $\Sigma_{a, n, t}$ and $\Sigma_{a', n', t'}$ produce distinct genus-rank functions if $a\neq a'$. In addition, these point-wise lower bounds coupled with the hypothesis that $\beta_2(X)<\infty$ guarantee that  each collection $\{\Sigma_{a, n, t}\mid n \in\mathbb{Z}^{> 0}, t\in\mathbb{R}^{\geq 0}\}$ for a fixed $a\in \mathbb{Z}^{> 0}$ produces only finitely many distinct genus-rank functions. 

Next, we turn to the Taylor invariant. Our goal is to find upper and lower bounds on each $\gamma(X_{\Sigma_{a, n, t}})$ that only depend on $n$, although our lower bounds will sometimes need to depend on $a$ as well. For each $n\in\mathbb{Z}^{> 0}$, consider the smooth structure $\Sigma_n'$ on $X$ that is obtained by end-summing $X_{\Sigma_\text{std}}$ with $\natural_n L$ along the ray $\gamma$ in $X_{\Sigma_\text{std}}$ and an arbitrary ray in $\natural_n L$. We will use the proof of Theorem 6.4\linebreak in \cite{Tay} to verify that each $\gamma(X_{\Sigma_{n}'})<\infty$. This theorem applies to $X_{\Sigma_{\text{std}}}$ because $H$ has no $3$-handles, $\beta_2(X)<\infty$, and $w_2(X)$ is compactly supported. Its proof produces smooth $4$-manifolds $M(\rho)$ for $\rho\in\mathbb{Z}^{> 0}$ that each have finite Taylor invariant. It is straight-forward to construct a smooth embedding of any $X_{\Sigma_n'}$ into some $M(\rho)$, sending  duals of $w_2(X)$ to duals of $w_2(M(\rho))$. As needed, this means that each $\gamma(X_{\Sigma_n'})\leq \gamma(M(\rho))<\infty$ for some $\rho\in\mathbb{Z}^{>0}$. Next, recall that every Casson handle admits a smooth, orientation-preserving embedding into the standard $2$-handle that preserves attaching regions. So each $X_{\Sigma_{a}}$ admits a smooth, orientation-preserving embedding into $X_{\Sigma_\text{std}}$ that sends $\gamma_a$ to $\gamma$. It follows that each $X_{\Sigma_{a, n, t}}$ admits a smooth, orientation-preserving embedding into $X_{\Sigma_n'}$ that maps duals of $w_2(X)$ to duals of $w_2(X)$. Therefore, each $\gamma(X_{\Sigma_{a, n, t}})\leq \gamma(X_{\Sigma_{n}'})<\infty$. In particular, the Taylor invariant realized by any given $\Sigma_{a, n, t}$ has a finite upper bound that is only dependent on the parameter $n\in\mathbb{Z}^{> 0}$. On the other hand, the Taylor invariant realized by each $\Sigma_{a, n, t}$ is bounded below by $\gamma(\natural_n {L}_{t_0})-d_a$, where $d_a$ is equal to the $\mathbb{Z}_2$-dimension of $H_1(F, \mathbb{Z}_2)$ for some smoothly embedded, closed surface $F$ in $X_{\Sigma_a}$ that is dual to $w_2(X)$. This lower bound depends on the choice of both $a,n \in\mathbb{Z}^{> 0}$, but it can be replaced by the lower bound $\gamma(\natural_n {L}_{t_0})$ when $X$ is spin. Property (\ref{inv}) of Theorem \ref{inStein} ensures that both lower bounds grow to be arbitrarily large if we let $n$ tend to infinity while fixing $a$. So we can conclude that any collection $\{\Sigma_{a, n, t}\mid t\in\mathbb{R}^{\geq 0}\}$ for a fixed $a, n\in\mathbb{Z}^{> 0}$ realizes only finitely many values of the Taylor invariant, and the smallest Taylor invariant realized by this collection grows to be arbitrarily large (but finite) as $n$ tends to infinity for a fixed $a$. If $X$ is spin, then any $\{\Sigma_{a, n, t}\mid a\in\mathbb{Z}^{>0} , t\in\mathbb{R}^{\geq 0}\}$ for a fixed $n\in\mathbb{Z}^{>0}$ also realizes only finitely many values of the Taylor invariant and the smallest Taylor invariant realized by this collection still grows to be arbitrarily large (but finite) as $n$ tends to infinity. 
 
The rest of this proof will be a counting argument that applies the observations from the previous two paragraphs to verify that the necessary invariants have indeed been realized, starting with the claims made in the original statement. For any fixed $a\in\mathbb{Z}^{> 0}, t\in\mathbb{R}^{\geq 0}$, the collection $\{\Sigma_{a, n, t}\mid n\in\mathbb{Z}^{> 0}\}$ produces only finitely many distinct genus-rank functions. So each contains smooth structures that realize infinitely many (arbitrarily large  and finite) values of the Taylor invariant but all produce the same genus-rank function. Since infinitely many distinct  genus-rank functions are produced in this way, the first sentence holds. Next, suppose that $X$ is spin. For any fixed $n\in\mathbb{Z}^{>0},t\in\mathbb{R}^{\geq 0}$, the collection $\{\Sigma_{a, n, t}\mid a\in\mathbb{Z}^{> 0}\}$ realizes only finitely many values of the Taylor invariant. So each contains smooth structures that produce infinitely many distinct genus-rank functions but all realize the same value of the Taylor invariant. Arbitrarily large (finite) values of the Taylor invariant are realized in this way. Therefore, the second sentence is also true. We have completed the proof of the original statement.   
 
Before proceeding to the addendum, we analyze the compact equivalence classes that are realized by these new smooth structures. For the remainder of this proof, we suppose that $\Sigma_\text{std}$ is compactly positive definite. We claim that $\Sigma_{a, n, t'}\leq \Sigma_{a, n, t}$ if and only if $t'\leq t$. Our argument will follow the proof of Lemma 7.3 in \cite{G5}. One direction is clear. To see the other direction, suppose to the contrary that some $\Sigma_{a, n, t'}\leq \Sigma_{a, n, t}$ for $t'>t$.  In particular, this means that $\natural_n{L}_t$ embeds into  $X_{\Sigma_{a, n, t}}$ by a smooth, orientation-preserving embedding whose image has compact closure. Then there exists smooth, compact, codimension-$0$ submanifolds $K_1\subseteq X_{\Sigma_a}$ and $K_2\subseteq \natural_n {L}_t$ with the property that $\natural_n{L_t}$ admits a smooth, orientation-preserving embedding into the result $K$ of boundary summing $K_1$ and $K_2$. Recall from above that each $X_{\Sigma_a}$ admits a smooth, orientation-preserving embedding into $X_{\Sigma_\text{std}}$, ensuring that $K_1\subseteq X_{\Sigma_a}$ admits a smooth, orientation-preserving embedding into some smooth, closed, simply-connected, positive definite $4$-manifold $P$. It then follows that $K$ admits a smooth, orientation-preserving embedding into $P\#(\natural_n {L}_t)$. Putting it all together, we've shown that $\natural_n{L_t}$ admits a smooth, orientation-preserving embedding into a smooth, compact, codimension-$0$ submanifold of $P\#(\natural_n{L_t})$. After reversing orientation, a standard argument (e.g. see \cite{D}) applying Taubes's extension of Donaldson Theory in \cite{Tau} and property (\ref{end}) from Theorem \ref{inStein} produces the necessary contradiction. So our claim holds. 

We are now prepared to find the smooth structures described in the addendum. Since any collection $\{\Sigma_{a, n, t}\mid t\in \mathbb{R}^{\geq 0}\}$ for a fixed $a, n\in\mathbb{Z}^{> 0}$ realizes only finitely many values of the Taylor invariant and produces only finitely many distinct genus-rank functions, each has a subcollection $\mathcal{C}_{a, n}$ of smooth structures that all realize the same genus-rank function and Taylor invariant but represent uncountably many distinct compact equivalences classes. Furthermore, each $\mathcal{C}_{a, n}$ can be chosen so that these compact equivalence classes have the order type of $\mathbb{R}^{\geq 0}$. For any fixed $a\in\mathbb{Z}^{>0}$, the subcollections $\mathcal{C}_{a, n}$ realize infinitely many (arbitrarily large and finite) values of the Taylor invariant as $n$ tends to infinity but only finitely many distinct genus-rank functions. So there are infinitely many distinct genus-rank functions that are each produced by smooth structures on $X$ realizing infinitely many (arbitrarily large and finite) values of the Taylor invariant, with each pair occurring for compact equivalence classes with the order type of $\mathbb{R}^{\geq 0}$. Suppose next that $X$ is spin. For each fixed $n\in\mathbb{Z}^{>0}$, the subcollections $\mathcal{C}_{a, n}$ realize infinitely many distinct genus-rank functions as $a$ tends to infinity but only finitely many values of the Taylor invariant.  Hence, some arbitrarily large (finite) value of the Taylor invariant is realized by smooth structures on $X$ that produce infinitely many distinct genus-rank functions, and again each pair occurs for compact equivalence classes with the order type of $\mathbb{R}^{\geq 0}$. This completes the proof of the addendum. 
\end{proof}

In the preceding proof, we ensured the necessary invariants were realized by producing sufficient bounds. As the next example illustrates, we can sometimes compute these invariants more explicitly. 
\begin{example}
\label{bundle}
 Let $X=S^2\times \mathbb{R}^2$. Following Theorem 3.10 in \cite{G5}, we can choose Stein-Casson smoothings $\{\Sigma_a\mid a\in\mathbb{Z}^{> 0}\}$ on $X$ so that that the first characteristic genus of the genus-rank function on each $X_{\Sigma_a}$ is $a$. Furthermore, the existence of a smoothly embedded, homologically essential surface of genus strictly smaller than $a$ in some $X_{\Sigma_a}$ would violate the adjunction inequality associated to the Stein structure on $X_{\Sigma_a}$. Now construct $\Sigma_{a, n, t}$ for each $a, n\in\mathbb{Z}^{> 0}, t\in\mathbb{R}^{\geq 0}$ as in the preceding proof. The first characteristic genus of the genus-rank function on each $X_{\Sigma_{a, n, t}}$ is equal to $a$ because each still satisfies this adjunction inequality. So the genus-rank function on each $X_{\Sigma_{a, n, t}}$ is completely determined by the value of $a$, and different values of $a$ will produce distinct genus-rank functions.  Observe next that each $X_{\Sigma_{a, n, t}}$ has the peculiar property it contains a smoothly embedded copy of $\natural_n{L}_t$ and also admits a smooth embedding into $\natural_n{L}_t$. This means that each $\gamma(X_{\Sigma_{a, n, t}})=\gamma(\natural_n{L}_t)$. Then there is some $t_n\in\mathbb{R}^{\geq 0}$ for each $n\in\mathbb{Z}^{>0}$ such that each $\gamma(X_{\Sigma_{a, n, t}})=\gamma(\natural_n{L})$ for all $t\geq t_n$, ensuring that the Taylor invariant realized by any $\Sigma_{a, n, t}$ for $t\geq t_n$ is completely determined by $n$. As in the preceding proof, the compact equivalence class of any $\Sigma_{a, n, t}$ and $\Sigma_{a, n, t'}$ are the same if and only if $t=t'$. In summary, we've found infinitely many genus-rank functions that are each produced by smooth structures on $X$ realizing the same infinitely many (arbitrarily large and finite) values of the Taylor invariant, and each pair occurs for uncountably many compact equivalence classes. Notice that we have significantly more control in this situation than we did in the preceding proof.  
\end{example}

The techniques used in the proof of Theorem \ref{TG} can be adapted to construct smooth structures that realize infinite Taylor invariant.
\begin{theorem}
\label{inf}
Suppose that $X$ is a open, connected, topological $4$-manifold that is the interior of an oriented handlebody $H$ with all indices $\leq 2$, $\beta_2(X)\neq 0$, and $w_2(X)\in H^2(X, \mathbb{Z}_2)$ is compactly supported.  Then there are smooth structures on $X$ that produce infinitely many genus-rank functions but all realize infinite Taylor invariant. If $\beta_2(X)=\infty$, then uncountably many genus-rank functions occur in this way. \end{theorem}

\begin{proof}
It follows immediately from Theorem 3.4 of \cite{G5} that there are Casson smoothings on $X$ producing infinitely many distinct genus-rank functions, and we can apply Theorem 3.5 of \cite{G5} to arrange for the cardinality to be uncountable when $\beta_2(X)=\infty$. Let ${L}$ be a large, oriented $\mathbb{R}^4$ from from Theorem \ref{inStein}. Then construct new smooth structures by end-summing each of these Casson smoothings on $X$ with $\natural_\infty {L}$. Each resulting smooth structure clearly realizes infinite Taylor invariant. As in the proof of Theorem \ref{TG} and Addendum \ref{ad}, these end-sums can be performed so that any two can still be distinguished by the genus-rank functions they produce. (These genus-rank functions cannot be distinguished by their first characteristic genera alone in the uncountable case, but the argument still goes through identically.) 
\end{proof}

\begin{example}
As in Example \ref{endsummingex}, let $X$ be the connected sum of infinitely many copies of $S^2\times S^2$ and let $\Sigma_{\text{std}}$ be its standard smooth structure. Again, perform these connected sums so that $X$ has a unique end. Then Theorem \ref{inf} defines smooth structures on $X$ that produce uncountably many distinct genus-rank functions but all realize infinite Taylor invariant. At least on the standard basis for $H_2(X)$, we can follow Example \ref{endsummingex} to arrange for the genus function from each resulting smooth structure to agree with the genus function from a Stein-Casson smoothing on $X$. However, it seems to be unknown if any Stein-Casson smoothing on $X$ also realizes infinite Taylor invariant. For example, consider $L_\infty$ from Corollary \ref{notU}. It is immediate from the construction of $L_\infty$ that it smoothly embed into $X_{\Sigma_\text{std}}$, but it is not clear if $L_\infty$ smoothly embeds into any Stein-Casson smoothing on $X$.  For comparison, note that Example \ref{endsummingex} ensures $U$ smoothly embeds into $X_{\Sigma_\text{std}}$ but does not smoothly embed into any Stein-Casson smoothing on $X$. 
\end{example}

Even if $w_2(X)\in H^2(X, \mathbb{Z}_2)$ is not compactly supported, we can still control the genus-rank function while realizing uncountably many distinct compact equivalence classes. However, the Taylor invariant will always equal $-\infty$ in this situation.
\begin{theorem}
\label{CG}
Suppose that $X$ is a open, connected, topological $4$-manifold that is the interior of an oriented handlebody $H$ with all indices $\leq 2$, $0<\beta_2(X)<\infty$, and the standard smooth structure  that $X$ inherits as a handlebody interior is compactly positive definite. Then there are smooth structures on $X$ that all produce the same genus-rank function but represent compact equivalence classes with the order type of $\mathbb{R}^{\geq 0}$, and infinitely many genus-rank functions occur in this way. 
\end{theorem}

\begin{proof}
Using some fixed $n\in\mathbb{Z}^{> 0}$, construct smooth structures $\{\Sigma_{a, n, t}\mid a\in\mathbb{Z}^{>0}, t\in\mathbb{R}^{\geq 0}\} $ on $X$ as in the proof of Theorem \ref{TG} and Addendum \ref{ad}. The bounds produced in that proof on the resulting genus-rank functions still hold. Similarly, the claims made about compact equivalence classes are also still true. Therefore, we can choose subcollections $\mathcal{C}_{a}\subseteq \{\Sigma_{a, n, t}\mid t\in\mathbb{R}^{\geq 0}\}$ for each $a\in\mathbb{Z}^{> 0}$ so that each contains smooth structures that all produce the same genus-rank function but represent compact equivalence classes with the order type of $\mathbb{R}^{\geq 0}$. Infinitely many  genus-rank functions occur in this way. 
\end{proof}

\begin{example}
Let $Y=\#m\mathbb{C}P^2$ for some $m\in\mathbb{Z}^{>0}$. Construct an open $4$-manifold $X$ by puncturing $Y$ and then end-summing with infinitely many orientable $\mathbb{R}^2$-bundles over nonorientable surfaces. Choose these $\mathbb{R}^2$-bundles so that each has positive, odd Euler number. Now $w_2(X)\in H^2(X, \mathbb{Z}_2)$ is not compactly supported, so we lose all control over the Taylor invariant. However, notice that each of these $\mathbb{R}^2$-bundles (equipped with their standard smooth structures) smoothly embeds into a finite connected sum of $\mathbb{C}P^2$'s. So the standard smooth structure on $X$ is compactly positive definite. Since $\beta_2(X)=m<\infty$ and this standard smooth structure is inherited from a handlebody with all indices $\leq 2$, Theorem \ref{CG} defines new smooth structures on $X$  that all produce the same genus-rank function but realize uncountably many compact equivalence classes. This can be done for infinitely many distinct genus-rank function. 
\end{example}

By forfeiting control over the genus-rank function, we can more carefully consider the relationship between compact equivalence classes and the Taylor invariant. 
\begin{theorem}
\label{CT}
Suppose that $X$ is an open, connected, topological $4$-manifold that is the interior of an oriented handlebody $H$ and the standard smooth structure that $X$ inherits as a handlebody interior is compactly positive definite. Then there are smooth structures on $X$ that represent compact equivalence classes with the order type of $\mathbb{R}^{\geq 0}$.  If $H$ has finitely many $3$-handles, $\beta_2(X)<\infty$, and $w_2(X)\in H^2(X, \mathbb{Z}_2)$ is compactly supported, then these can be chosen to all realize the same arbitrarily large (finite) value of the Taylor invariant.
\end{theorem}

\begin{proof}
 For some fixed $a\in\mathbb{Z}^{>0}$, define a collection $\{\Sigma_{a, n, t}\mid n \in\mathbb{Z}^{> 0}, t\in\mathbb{R}^{\geq 0}\}$ of smooth structures on $X$ using the procedure in the proof of Theorem \ref{TG} and Addendum \ref{ad} but with the modification that $\Sigma_a$ is simply this standard smooth structure on $X$. The claims made in that proof about the resulting compact equivalence classes still hold. Thus, there are subcollections $\mathcal{C}_n\subseteq \{\Sigma_{a, n, t}\mid t\in\mathbb{R}^{\geq 0}\}$ for each $n \in\mathbb{Z}^{> 0}$ that realize compact equivalence classes with the order type of $\mathbb{R}^{\geq 0}$. Suppose next that $H$ has finitely many $3$-handles, $\beta_2(X)<\infty$, and $w_2(X)\in H^2(X, \mathbb{Z}_2)$ is compactly supported. Since Theorem 6.4 from \cite{Tay} still applies, the bounds on the Taylor invariant that are used in the proof of Theorem \ref{TG} and Addendum \ref{ad} also still hold. Therefore, these subcollections can be chosen so that the smooth structures in each $\mathcal{C}_n$ all realize the same value of the Taylor invariant. As needed, infinitely many (arbitrarily large and finite) values of the Taylor invariant occur in this way. 
\end{proof}

\begin{remark}
The requirement that $H$ has finitely many $3$-handles could be replaced by the \textit{few essential 3-handles} condition from \cite{Tay}. 
\end{remark}

\begin{example}
 Let $X$ be the interior of the handlebody obtained by boundary summing together $\{H_i\}_{i=0}^\infty$ with each $H_i$ constructed by attaching a $1$-framed $2$-handle to $B^4\cup (1\text{-handle})$ along the simplest knot that links the corresponding dotted circle $2i+1$ times. Observe that $X$ is compactly positive definite because $\mathbb{C}P^2$ can be obtained from any $H_i$ by attaching a $2$- and $4$-handle. Since $H_2(X, \mathbb{Z}_2)=H_2(X)=0$, it is clear that $X$ satisfies all hypotheses from Theorem \ref{CT}. Therefore, there are smooth structures on $X$ representing uncountably many compact equivalence classes that all realize the same arbitrarily large (finite) value of the Taylor invariant. Notice that the genus-rank function provides no information in this case. 
\end{example}

We can sometimes improve the applications presented in this section by realizing each combination of invariants for uncountably many distinct diffeomorphism types. This is achieved using the method introduced in Section 7 of \cite{G5}, which produces uncountable families of smooth structures by starting with a compactly negative definite smooth structure on an open, topological $4$-manifold and then end-summing with small $\mathbb{R}^4$'s. Our argument relies on the fact that the $\mathbb{R}^4$'s used in the proof of each theorem above embed into smooth, closed, simply connected, negative definite $4$-manifolds. 
\begin{theorem}
\label{small} Suppose that $X$ is an open, connected, topological $4$-manifold that satisfies the hypotheses of Theorem \ref{TG}, Addendum \ref{ad},  Theorem \ref{inf}, Theorem \ref{CG}, or Theorem \ref{CT} as the interior of a handlebody $H$. If also the standard smooth structure that $X$ inherits as the interior of $H$ is compactly negative definite, then each combination of genus-rank function, Taylor invariant, and compact equivalence class realized by this theorem or addendum can be chosen so that it occurs for uncountably many diffeomorphism classes of smooth structures on $X$. 
\end{theorem}

\begin{proof}
  Let $\Sigma$ be any smooth structure on $X$ that is constructed by Theorem \ref{TG}, Addendum \ref{ad}, Theorem \ref{inf}, Theorem \ref{CG}, or Theorem \ref{CT} applied to $X$ as the interior of $H$. Recall that $\Sigma$ is produced by end-summing some smoothing $\Sigma_a$ on $X$ with some $\natural_n{L}_t$ for $n\in\mathbb{Z}^{> 0}\cup\{\infty\}$, where $L_t$ is an $\mathbb{R}^4$ contained in some $L$ produced by Theorem \ref{inStein}.  The smooth structure $\Sigma_a$ is either a Casson smoothing corresponding to $H$ or the standard smoothing that $X$ inherits as a handlebody interior. In either case, $X_{\Sigma_a}$ admits a smooth, orientation-preserving embedding into this standard smooth structure on $X$. It follows from property (\ref{neg}) of Theorem \ref{inStein} that $\natural_n{L}_t$ admits a smooth, orientation-preserving embedding into $\#n\overline{\mathbb{C}P^2}$. Using our new hypothesis, we can now easily conclude that $\Sigma$ is compactly negative definite. Hence, $\Sigma$ satisfies the hypothesis of Theorem 7.1 from \cite{G5}. So we obtain uncountably many diffeomorphism classes of smooth structures on $X$ by applying that theorem to $\Sigma$. Its proof ensures that every smooth $4$-manifold obtained by equipping $X$ with one of these new smooth structures admits a smooth embedding into $X_\Sigma$ and visa versa. Both embeddings are topologically isotopic to the identity, so it is clear that each of these new smooth structures realizes that same genus-rank function, Taylor invariant, and compact equivalence class as $\Sigma$. 
\end{proof}

\begin{example}
This example provides a construction of compact, spin $4$-dimensional handlebodies whose interiors are both compactly positive definite and compactly negative definite but do not smoothly embed into $S^4$, each of which has all indices $\leq 2$ and nonzero second betti number. Suppose that $K$ is an amphichiral knot in $S^3=\partial B^4$ with unknotting number and $4$-ball genus both equal to $1$, such as the figure-8 knot or the $6_3$ knot. Then let $H_K$ be the handlebody obtained by attaching a $0$-framed $2$-handle to $B^4$ along $K$. Notice immediately that $H_K$ does not smoothly embed into $S^4$ because $K$ is not smoothly slice. Our requirement on the unknotting number ensures that $K$ bounds a smoothly immersed disk with a single double point, and we can arrange for either sign on this double point. First consider the case when the double point is positive. Resolve this singularity using $\mathbb{C}P^2$, so that $K$ bounds a smoothly embedded disk $D$ in $B^4\#\mathbb{C}P^2$ whose normal framing induces the $0$-framing on $K$. Capping off with a $4$-handle produces $\mathbb{C}P^2$, and we can see $H_K$ smoothly embedded in $\mathbb{C}P^2$ as the union of this $4$-handle and a tubular neighborhood of $D$. Starting with a negative double point and undergoing the same procedure instead with $\overline{\mathbb{C}P^2}$ produces a smooth embedding of $H_K$ into $\overline{\mathbb{C}P^2}$. Each of these embeddings is actually orientation-reversing, but $H_K$ still admits a smooth, orientation-preserving embedding into both $\mathbb{C}P^2$ and $\overline{\mathbb{C}P^2}$ (either by reversing the orientation on both target manifolds or by observing that $H_K$ admits an orientation-reversing self-diffeomorphism). So the oriented, open, topological $4$-manifold $X_K$ obtained as the interior of the handlebody $H_K$ inherits a smooth structure that is both compactly positive definite and compactly negative definite. Since $\beta_2(X_K)=1$ and $w_2(X_K)=0\in H^2(X_K, \mathbb{Z}_2)$, it follows that $X_K$ satisfies all hypotheses of both Theorem \ref{TG}\linebreak  and Addendum \ref{ad} (and also Theorem \ref{inf}, Theorem \ref{CG}, and Theorem \ref{CT}) as the interior of $H_K$. Therefore, we can construct smooth structures on $X_K$ with independent control over the resulting genus-rank functions, Taylor invariants, and compact equivalence classes. Furthermore, Theorem \ref{small} finds uncountably many distinct diffeomorphism classes of smooth structures on $X_K$ realizing any triple that occurs in this way. This procedure can clearly be generalized to more complicated knots and links with similar unknotting properties.  
\end{example}

We conclude this section with what seems to be the ``largest" family of exotic $\mathbb{R}^4$'s, but note that this corollary could instead be obtained by combining previously known results about $\mathbb{R}^4$'s in the appropriate way. As in Corollary \ref{notStein}, none of these $\mathbb{R}^4$'s admit a Stein structure themselves. However, for a fixed value of the Taylor invariant, it is clear that all of the representatives produced by this corollary do smoothly embed into some fixed Stein surface. Furthermore, this Stein surface can be chosen to be compact when the Taylor invariant is finite.  
\begin{cor}
\label{R4's}
There are uncountably many compact equivalence classes on $\mathbb{R}^4$ that each contain uncountably many diffeomorphism types, and these compact equivalence classes all realize the same arbitrarily large (finite) value of the Taylor invariant. Additionally, there are uncountably many compact equivalence classes on $\mathbb{R}^4$ with infinite Taylor invariant and at least one contains uncountably many diffeomorphism types.  
\end{cor}

\begin{proof}
The first claims follows immediately by applying Theorem \ref{CT} and Theorem \ref{small} to $X=\mathbb{R}^4$. Next, consider $L_\infty$ from Corollary \ref{notU}. Since $L_\infty$ is compactly negative definite, Theorem 7.1 of \cite{G5} finds uncountably many $\mathbb{R}^4$'s with infinite Taylor invariant in the same compact equivalence class as $L_\infty$. To realize uncountably many compact equivalence classes with infinite Taylor invariant, choose a handlebody (which will have infinitely many $3$-handles) from which $L_\infty$ inherits its smooth structure, reverse orientation, and then apply the first part of Theorem \ref{CT} to the result. Each resulting compact equivalence class has infinite Taylor invariant because each has a representative that contains $\overline{L_\infty}$. 
\end{proof}

\begin{remark}
 The theorems in this section could be stated in more generality. First of all, if a given homeomorphism type does not satisfy the hypotheses of some theorem above but has a sufficiently nice finite cover that does, then these methods can be adapted as in \cite{G5} to produce smooth structures whose diffeomorphism types are distinguishable by the invariants realized by their lifts to this cover. As above, we can independently control the values of these invariants. Secondly, following Corollary 7.4\linebreak of \cite{G5}, the compactly negative definite or compactly positive definite conditions only need to hold on a closed, noncompact subset $Y$ that has a smooth, compact $3$-manifold boundary. After throwing away duplicates, there are still uncountably many resulting smooth structures and they are distinguishable by the diffeomorphism types of their restrictions to $Y$. For example, we might construct uncountably many smooth structures that can be differentiated by the compact equivalence class of their restriction to $Y$. 
\end{remark}

\bibliographystyle{alphanum}

\vspace{1em}\small{
\noindent\textsc{Department of Mathematics, The University of Texas at Austin, 78712}
\vspace{.3em}

\noindent\textit{Email Address: }\href{mailto:jbennett@math.utexas.edu}{jbennett@math.utexas.edu}
\vspace{.3em}

\noindent\textit{Website: }\href{http://www.ma.utexas.edu/~jbennett/}{www.ma.utexas.edu/$\sim$jbennett/}}

\end{document}